\documentclass[12pt]{amsart}

\usepackage{epstopdf}
\usepackage[caption=false]{subfig}

\usepackage[usenames]{color}
\input xypic

\usepackage[colorlinks=false, linkcolor=blue]{hyperref}
\usepackage{verbatim}
\usepackage{url}
\usepackage{bm}
\usepackage{tikz}
\usepackage{graphicx}
\usepackage{color}
\usepackage{nicematrix}
\usepackage{dsfont}
\usepackage{amsmath}
\usepackage{amsfonts}
\usepackage{amssymb}
\usepackage{amsthm}
\usepackage{newlfont}
\usepackage{url}
\usepackage{bm}
\usepackage{chngcntr}
\usepackage[utf8]{inputenc} 
\usepackage[T1]{fontenc}
\usepackage{listings}
\usepackage{rotating} 

\usepackage[a4paper,margin=2.5cm]{geometry}

\usepackage{algorithm}
\usepackage[noend]{algpseudocode}

\usepackage{tabularx}
\usepackage{xltabular}
\usepackage{array}
\usepackage{booktabs}
\usepackage{ragged2e}

\newcolumntype{Y}{>{\RaggedRight\arraybackslash}X}

\algblock{Input}{EndInput}
\algnotext{EndInput}
\algblock{Output}{EndOutput}
\algnotext{EndOutput}
\newcommand{\Desc}[2]{\State \makebox[2em][l]{#1} #2}

\usepackage[numbers,sort&compress]{natbib}
\bibpunct[, ]{[}{]}{,}{n}{,}{,}

\theoremstyle{plain}
\newtheorem{theorem}{Theorem}[section]
\newtheorem{lemma}[theorem]{Lemma}

\newtheorem{proposition}[theorem]{Proposition}
\newtheorem{conjecture}[theorem]{Conjecture}

\theoremstyle{definition}
\newtheorem{definition}[theorem]{Definition}
\newtheorem{example}[theorem]{Example}

\theoremstyle{remark}
\newtheorem{remark}{Remark}

\newcommand{\tm}{$\mathbf{t}$-}
\newcommand{\Fq}{\mathbb{F}_q}
\newcommand{\Ga}{\mathbb{G}_{a,K}}

\newcommand{\rk}{\mathrm{rk}}
\newcommand{\Mat}{\mathrm{Mat}}
\newcommand{\Lie}{\mathrm{Lie}}

\newcommand{\End}{\mathrm{End}}
\newcommand{\Ext}{\mathrm{Ext}}
\newcommand{\Der}{\mathrm{Der}}
\newcommand{\Derin}{\mathrm{Der_{in} }}

\newcommand{\Extt}{\Ext_\tau^1}

\newcommand{\podwzorem}[2]{\underbrace{#1}\limits_{#2}}

\begin{document}


\title{ Duality for \tm modules: The Difficult Cases}

\author{Dawid Edmund K{\k e}dzierski}

\address{ Institute of Mathematics, University of Szczecin,\\ ul. Wielkopolska 15, 70-451 Szczecin, Poland\\
}
\email{dawidedmundkedzierski@gmail.com}

\author{Piotr Kraso{\'n}}

\address{ Institute of Mathematics, University of Szczecin, \\ ul. Wielkopolska 15, 70-451 Szczecin, Poland\\
}
\email{piotrkras26@gmail.com }

\keywords{
Drinfeld modules; \tm modules; triangular \tm modules; 
duality of \tm modules; Weil-Bersotti formula; Cartier-Nishi theorem
}

\subjclass{Mathematics Subject Classification 2020: 11G09; 68W30; 11R58; 16S36; 11Y40}

\maketitle

\begin{abstract}
This paper continues our previous work on duality for Anderson t-modules. We study two-dimensional triangular \tm modules in a specific reduced form. Computer-assisted symbolic computations with matrices over a skew field revealed a reduction pattern for \tm modules satisfying the ALD condition. Using this pattern, we prove that such t-modules are isomorphic to their double duals and extend the validity of the Cartier–Nishi theorem and the Weil–Barsotti formula to a substantially broader class of t-modules.
\end{abstract}

\section{Introduction}

Drinfeld modules and their higher dimensional generalizations, since their invention, have been central objects of intensive studies in number theory.  In recent years, many algorithmic methods have been developed for the \tm modules  and the associated  \tm motives and \tm co-motives e.g.  \cite{kuhnpink}, \cite{Mau1}, \cite{arm}. This creates a natural environment for experimental mathematics. The present paper represents one of the first experimental approaches to problems in function field arithmetic over global function fields.
We address the concept of duality for \tm modules. The duality for Drinfeld modules, in terms of algebraic geometry, was investigated in \cite{ta}. In \cite{pr}, the authors defined duality for Drinfeld modules of rank at least two by means of the $\Ext^1_{\tau}$-functor and proved the Weil-Barsotti formula, as well as the Cartier-Nishi theorem \cite[Theorem 1.1]{pr}. In \cite{kk04}, we showed that, for a large class of  pairs of  \tm modules $\Phi$ and $\Psi$,  $\Ext^1(\Phi,\Psi)$ can be endowed with a natural structure of a \tm module. This was the starting point for our further investigation of the duality.   In \cite{gkk} we developed an algorithm, called the \tm reduction, which implements these  results. 
Computing endomorphism rings of Drinfeld modules is a  highly nontrivial problem (see \cite{kuhnpink} and \cite{gp}). For a \tm module ${\Upsilon}$ computing a double dual module ${\Upsilon^\vee}^\vee$ and finding that ${\Upsilon^\vee}^\vee\cong {\Upsilon}$
 is by no means easy.  Notice, that one cannot expect the Cartier-Nishi theorem and the Weil-Barsotti formula to hold for every \tm module $\Upsilon$ (see\cite[Example 4.1]{kk24}). In \cite{kk24} we proved these for $\Upsilon$ a non-nilpotent, stictly pure \tm module of degree at least two. Since the proof was  done in the language of biderivations it required careful  reductions by inner biderivations. In \cite{kk25} we introduced triangular \tm modules i.e. those with the composition series 
 having Drinfeld modules as  $\tau$-simple sub-quotients. For such modules allowing for low degree biderivations we were able to prove the Cartier-Nishi theorem and the Weil-Barsotti formula. In this paper, we deal with the most difficult cases where substantial coefficient swell occurs. This phenomenon constitutes a serious computational obstacle. To determine whether any patterns emerge in reductions leading to the isomorphism ${\Upsilon^\vee}^\vee\cong\Upsilon $, we performed a large number of computer experiments. We restricted our attention to triangular \tm modules of dimension two. These experiments enabled us to extend the aforementioned Cartier-Nishi theorem and Weil-Barsotti formula to  two-dimensional \tm modules  allowing for ${\mathrm{ALD}}$ biderivations. 
 The structure of the paper is as follows. In Section~\ref{prem}, we briefly review the necessary definitions. In Section~\ref{trian}, we introduce triangular modules, the main objects of study in this paper. We then formulate and discuss a general conjecture (cf. Conjecture~\ref{conjecture}).
 We then restrict our attention to two-dimensional triangular \tm modules. In particular, we emphasize that in the more difficult range
$\rk\psi_1 > \deg_\tau \delta \geq \rk\psi_2,$
the isomorphisms
$\Upsilon \cong {\Upsilon^{\vee}}^{\vee}$
are no longer given by constants; instead, they are realized by skew polynomials.
In Section~\ref{sec:computational_method}, we describe our computational methods in detail. In particular, we present the algorithms developed for this purpose.
Section~\ref{symb} is devoted to the complexity analysis of the computation of dual and double dual modules in the aforementioned difficult cases of two-dimensional triangular \tm modules.
In Section~\ref{poly}, we present the results of several computer-assisted experiments. Based on these results, we formulate a conjecture for two-dimensional triangular \tm modules and state the main theorem of this paper. In Section~\ref{sec:explicit-dual}, we compute the dual module explicitly (cf. Theorem~\ref{thm:postac_duala}) and prove an auxiliary proposition concerning the degree of a dual \tm module (cf. Proposition~\ref{prop:deg_dual_module}).
Sections~\ref{pfthm} and~\ref{pflemma} are devoted to the proofs of the main theorem and an auxiliary lemma, respectively.
Finally, in Section~\ref{disc}, we discuss our results and outline possible directions for future generalizations.

\section{Preliminaries}\label{prem}
In this section, we recall only the basic notions concerning \tm-modules, restricting ourselves to the algebraic side of the theory.  
For comprehensive introduction and further details, the reader is referred to the excellent sources \cite{g96}, \cite{bp20}, \cite{th04}, \cite{f13}.

Let $A=\Fq[t]$ be the polynomial ring over the finite field $\Fq$ with $q=p^{s}$ elements and let $K$ be a field of characteristic $p$ equipped with an $\Fq$–algebra homomorphism  $\iota:A\to K.$  We write $\theta := \iota(t)$.
The endomorphism ring $\End(\Ga)$, where $\Ga$ denotes the additive algebraic group over $K$, is the skew-polynomial ring $K\{\tau\}$.  
Here $\tau$ is the $q$–Frobenius endomorphism $x\mapsto x^{q}$, and therefore  the commutation relation
\[
\tau x = x^{q}\tau \qquad (x\in K)
\]
holds.
A definition of a $d$-dimensional \tm module was given by G.~Anderson \cite{a}.  It is an algebraic group $E$ defined over $K$ and isomorphic over $K$ to $\Ga^{d}$, together with a choice of an $\Fq$–linear endomorphism $l:E\to E$ such that  
$
\partial(l-\theta)^{n}\Lie(E)=0
$
for all sufficiently large $n$, where $\partial(-)$ denotes the differential of an endomorphism of algebraic groups.  
After fixing an isomorphism $E\cong\Ga^{d}$, the endomorphism $l$ determines an $\Fq$–algebra homomorphism
$\Phi:\Fq[t]\longrightarrow \Mat_{d}(K\{\tau\})$
such that the element $\Phi_{t}:=\Phi(t)$, regarded as a skew-polynomial in~$\tau$ with coefficients in $\Mat_d(K)$, has the form
\begin{equation}\label{eq:def:t-modul}
\Phi_{t} = (\theta I_d + N_{\Phi})\tau^{0} + M_{1}\tau + \dots + M_{r}\tau^{r},
\end{equation}
where $I_d$ is the identity matrix and $N_\Phi$ is a nilpotent matrix.
A \tm-module $\Phi$ is said to have \emph{no nilpotence} if $N_\Phi=0$ in \eqref{eq:def:t-modul}.  If the dimension $d=1,$ then a \tm module is called a Drinfeld module. Drinfeld modules were introduced in \cite{d74}.
For a general \tm module, there are well-defined notions of rank and degree; see \cite{bp20}. 
The category of \tm-modules is additive but not abelian; see \cite{kk04}.   Nevertheless,  there exists a well-defined notion of an exact sequence. Consequently, one can consider extension groups of pairs of \tm modules. The study of the such extension groups was initiated in \cite{pr} and continued in \cite{kk24,gkk,kk25}.

An extension of a \tm-module $\Phi:\Fq[t]\to\Mat_{d}(K\{\tau\})$ by a \tm-module  
$\Psi:\Fq[t]\to\Mat_{e}(K\{\tau\})$ can be represented by a {\emph biderivation}, that is, an $\mathbb{F}_q$-linear map 
$\delta:\Fq[t]\to \Mat_{e\times d}(K\{\tau\})$ satisfying
$
\delta(ab) = \Psi(a)\delta(b) + \delta(a)\Phi(b),$ for 
$
a,b\in\Fq[t].
$
The $\Fq$–vector space of all biderivations is denoted by $\Der(\Phi,\Psi)$. 
The map $\delta\mapsto\delta(t)$ induces an $\Fq$–linear isomorphism  
$
\Der(\Phi,\Psi) \;\cong\; \Mat_{e\times d}(K\{\tau\}).
$
Let  
$
\delta^{(-)}:\Mat_{e\times d}(K\{\tau\})\to\Der(\Phi,\Psi)
$
be the $\Fq$–linear map defined by
$\delta^{(U)}(a)=U\Phi_{a}-\Psi_{a}U$, for $a\in\Fq[t]$, $U\in\Mat_{e\times d}(K\{\tau\})$.
Its image, consisting of \emph{inner biderivations}, is denoted by $\Derin(\Phi,\Psi)$. Henceforth, we write $\delta^{(U)}_{\Phi,\Psi}$ to denote the element $\delta^{(U)}_{\Phi,\Psi}\in\Derin(\Phi,\Psi)$.
We have the $\Fq[t]$–module isomorphism 
$
\Ext^{1}_{\tau}(\Phi,\Psi)
\;\cong\;
\Der(\Phi,\Psi)/\Derin(\Phi,\Psi).
$
Recall that the structure of the $\Fq[t]$ -module in this quotient is given by
\begin{equation}\label{eq:wzor_na_mnozenie_w_Ext}
  a*\big(\delta+\Derin(\Phi,\Psi)\big)
:=\Psi_{a}\delta+\Derin(\Phi,\Psi).  
\end{equation}

For $V\in\Mat_{n_{1}\times n_{2}}(K\{\tau\})$ let $\partial V$ denote the constant term of~$V$ viewed as a matrix skew polynomial in~$\tau$.  
For \tm-modules $\Phi$ and $\Psi$ define
$\Der_{0}(\Phi,\Psi)
=
\{\delta\in\Der(\Phi,\Psi)\mid \partial\delta(t)=0\},
$
and
$
\Der_{0,\mathrm{in}}(\Phi,\Psi)
=
\{\delta\in\Derin(\Phi,\Psi)\mid \partial\delta(t)=0\}.$
 We then set
$
\Ext^{1}_{0,\tau}(\Phi,\Psi)
:=
\Der_{0}(\Phi,\Psi)/\Der_{0,\mathrm{in}}(\Phi,\Psi).
$
In general, $\Ext^{1}_{0,\tau}(\Phi,\Psi)$ is merely an $\Fq[t]$–module, although in certain situations this structure arises from a genuine \tm module; see \cite{pr}, \cite{kk04}, and \cite{gkk}.
For a \tm-module $\Phi$ we define its \emph{dual} by
$
\Phi^{\vee}
=
\operatorname{Ext}^{1}_{0,\tau}(\Phi,C),$
where $C$ denotes the Carlitz module, i.e., the Drinfeld module defined by
$ C_t=\theta+\tau.$

\section{Triangular \tm modules of dimension two}\label{trian}

The notion of a composition series for \tm-modules was adapted in \cite{gkk}; see \cite[Theorem~5.5]{gkk}.  
A \tm-module is called \emph{triangular} if it admits a composition series whose successive quotients are Drinfeld modules.  
The structure and properties of triangular \tm-modules were studied extensively in \cite{kk25}.  
In this article, we focus on triangular \tm-modules of dimension~$2$, namely those arising as the middle term of  a short exact sequence
$
0\longrightarrow \psi_{1}\longrightarrow \Upsilon\longrightarrow \psi_{2}\longrightarrow 0.
$
After a suitable choice of coordinates on $\Upsilon$, such a module acquires a triangular form
\begin{equation}\label{trojkatnymoduldim2}
\Upsilon
=\begin{pNiceArray}{cccccc}[margin=5pt]
          \Block[]{1-1}<>{\psi_2}&0 \\[2pt] 
         \Block[]{1-1}<>{\delta}& 
         \Block[]{1-1}<>{\psi_1}
\end{pNiceArray}.
\end{equation}
Thus, each such module is uniquely determined by a biderivation  
$\delta:\Fq[t]\to K\{\tau\}$, or equivalently by the skew polynomial $\delta(t)\in K\{\tau\}$. We denote the corresponding \tm module by  $\Upsilon=\Upsilon(\psi_1,\psi_2,\delta)$.

The main objective of this article is to investigate the following conjecture.
\begin{conjecture}\label{conjecture}
Let $\Upsilon$ be a triangular \tm-module.  
Then ${\Upsilon^{\vee}}^{\vee}\cong \Upsilon$, and moreover there exists an exact sequence of \tm-modules
\[
0\longrightarrow \Upsilon
\longrightarrow
\Ext^{1}_{\tau}(\Upsilon^{\vee},C)
\longrightarrow
\Ga^{s}
\longrightarrow 0,
\]
for some natural number $s$ depending on~$\Upsilon$.
\end{conjecture}

This conjecture is known to hold for Drinfeld modules of rank at least~$2$ \cite{pr}, for strictly pure \tm-modules of degree at least~$2$ without nilpotence  \cite{kk24}, and for triangular \tm-modules allowing for $\mathrm{LD}$–biderivations, without nilpotents and with quotient modules of rank at least~$2$ \cite{kk25}.  In the latter two settings, the conjecture fails once the assumption of non-nilpotence is dropped; see \cite[Example~4.1]{kk24} and \cite[Example~9.1]{kk25}. 

Motivated by these counterexamples, we restrict our attention to triangular \tm modules of the form \eqref{trojkatnymoduldim2} without nilpotence and with $\rk\psi_{i}\ge 2$, while allowing biderivations beyond the class of LD-biderivations.

Note that for triangular \tm-modules, \cite[Theorem~9.1]{kk25} implies that the condition  
${\Upsilon^{\vee}}^{\vee}\cong\Upsilon$ already yields the full conclusion of Conjecture~\ref{conjecture}.  
Moreover, the same theorem gives that
\begin{equation}\label{doubledualtrojkatnymoduldim2}
{\Upsilon^{\vee}}^{\vee}
=\begin{pNiceArray}{cc}[margin=5pt]
          \Block[]{1-1}<>{{\psi_2^\vee}^\vee}&0 \\[2pt] 
         \Block[]{1-1}<>{ {\delta^\vee}^\vee }& 
         \Block[]{1-1}<>{{\psi_1^\vee}^\vee}
\end{pNiceArray}.
\end{equation}

Let $r_{i} := \rk\psi_{i}$ for $i=1,2$, and let $a_{r_1}$ and $b_{r_{2}}$ denote the leading coefficients of $\psi_{1}$ and $\psi_{2}$, respectively. 
 Recall from the proof of \cite[Theorem~3.4]{kk24} that
\[
{\psi_{1}^{\vee}}^{\vee}
=
\frac{1}{a_{r_{1}}}\cdot \psi_{1}\cdot a_{r_{1}},
\qquad
{\psi_{2}^{\vee}}^{\vee}
=
\frac{1}{b_{r_{2}}}\cdot \psi_{2}\cdot b_{r_{2}}.
\]
 To recover the original subquotients in ${\Upsilon^{\vee}}^{\vee}$, we consider the conjugate of ${\Upsilon^{\vee}}^{\vee}$ by the diagonal matrix  
$D = \mathrm{diag}(b_{r_{2}},a_{r_{1}})$. 
Thus,
\[
{}^{D}{\Upsilon^{\vee}}^{\vee}
:=
D\,{\Upsilon^{\vee}}^{\vee}\,D^{-1}
=\begin{pNiceArray}{cc}[margin=5pt]
          \Block[]{1-1}<>{{\psi_2^\vee}^\vee}&0 \\[2pt] 
         \Block[]{1-1}<>{ a_{r_{1}}\cdot{\delta^{\vee}}^{\vee}\cdot\frac{1}{b_{r_{2}}} }& 
         \Block[]{1-1}<>{{\psi_1^\vee}^\vee}
\end{pNiceArray}.
\]

Therefore, the existence of an isomorphism $\Upsilon\cong{\Upsilon^{\vee}}^{\vee}$ is equivalent to the condition
\begin{equation}\label{eq:izo_warunek}
a_{r_{1}}\cdot{\delta^{\vee}}^{\vee}\cdot\frac{1}{b_{r_{2}}} - \delta
\;\in\;
\Derin(\psi_{2},\psi_{1}).
\end{equation}

It is worth noting that, in all previously studied cases, the isomorphism  
$\Upsilon\cong{\Upsilon^{\vee}}^{\vee}$ was realized by a matrix with entries in the field~$K$.  
 As we shall see in later examples lying  beyond the class of modules allowing for $\mathrm{LD}$–biderivations, this is no longer the case: the relevant isomorphisms are represented  by matrices whose entries are genuine skew polynomials rather than constants.  
This phenomenon substantially increases the complexity of the problem.

Observe  that, if $\rk\psi_1 \leq \rk\psi_2$ then the triangular \tm-module
$\Upsilon(\psi_1,\psi_2,\delta)$ allows for  $\mathrm{LD}$-biderivations (see
\cite[Corollaries 4.3 and 4.4]{kk25}). Consequently, our attention is restricted to the case  $\rk\psi_1 > \rk\psi_2$.
Furthermore, every triangular \tm module admits a reduced form [KK25, Theorem 4.1], for which
$\deg_\tau \delta < \max\{\rk\psi_1,\rk\psi_2\}$. Hence, throughout this paper we consider only examples satisfying
\begin{equation}
\label{eq:warunek_na_stopnie}
    \rk\psi_1 > \deg_\tau \delta \geq \rk\psi_2.
\end{equation}

\section{Computational method}\label{sec:computational_method}

The starting point of our approach is the recent work [GKK24], in which a \tm reduction algorithm was developed for determining the \tm module structure on the extension groups $\Ext^{1}_{\tau}(\Phi,\Psi),$ for suitable pairs of \tm modules $\Phi$ and $\Psi.$
In the present paper, we use the \tm reduction algorithm in two distinct ways. First, we apply it to compute the dual of a triangular \tm module $\Upsilon.$ The computation of $\Upsilon^\vee$ is obtained as a specialization of [GKK24, Algorithm 6] to the case $\Psi = C$, carried out only for biderivations whose differentials vanish. Second, we extend the \tm reduction algorithm to a new setting, namely the computation of the double dual biderivation ${\delta^\vee}^\vee$ in the case where the dual \tm module has an upper-triangular form.
\[
\Upsilon^{\vee}
=\begin{pNiceArray}{cc}[margin=5pt]
         \psi_{2}^{\vee} & \delta^{\vee} \\
  0               & \psi_{1}^{\vee}
\end{pNiceArray}.
\]

This case is substantially more delicate than those previously considered in \cite{gkk}.
In our pseudo-codes we use the following notation from \cite{gkk}:
\begin{itemize}
\item $x^{(k)}=x^{q^k}$
    \item $E_{i\times j}$ -- $n \times m$ matrix with 1 in i-th row and j-th column
 and $0$ everywhere else;
 \item $C=\theta+\tau$ -- Carlitz module;
    \item $\textsc{Deg(W)}$ -- degree of  a $\tau$- polynomial $W$;
\item $\textsc{Dim(F)}$ -- dimension of a t-module $F$ ;
\item $\textsc{Rows(M)}$ -- number of rows of a matrix $M$ 
\item $\textsc{Cols(M)}$ -- number of columns of a matrix $M$ 
\item $\textsc{Coefficient(w, n)}$ -- coefficient $a_{\textproc{n}}$ of a $K\{\tau\}$ polynomial $\textproc{w} = a_m {\tau}^m + ... + a_0$;
\item $\textproc{Substitute(expression, s, t)}$ -- 
replaces symbol $\textproc{s}$ in $\textproc{expression}$
 with symbol $\textproc{t}$ (sometimes we substitute for any expression, then $\textproc{(\_)}$ stands for any expression);
\item $\textproc{M}[i,j]$ - element $M_{i,j}$ of a matrix $\textproc{M}$; 
\item $\begin{bmatrix}
    \boldsymbol{V_{1}} & ... & \boldsymbol{V_{k}}
    \end{bmatrix}$ - matrix with columns $\boldsymbol{V_{i}}$
\item $\textsc{pMult}(u,v)$ -- computes the product of two skew polynomials $u, v\in K\{\tau\}$; see \cite[Algorithm 1]{gkk}.
\item $\textsc{TMult}(A,B)$ -- computes the product of matrices A and B with entries in $K\{\tau\}$; see \cite[Algorithm 1]{gkk}.
\item  $\textsc{Reduce2}(V,\Psi,\Psi)$ -- reduces the skew polynomials in the matrix $V$ using inner biderivations of the form $\delta_{\Phi,\Psi}^{(U)}$; as a result, the degree of the polynomial in the $i\times j$-entry is less than the degree of $\Phi[j,j]$; see \cite[Algorithm 5]{gkk}.
\end{itemize}

\begin{algorithm}[H]
\caption{}\label{alg:helper3}
\begin{algorithmic}[1]
\Function{CoefficientFormDual}{$\boldsymbol{V}, {d}$}

\Comment{lists all coefficients from degree one up to degree defined by matrix $d$ for each $\tau$ polynomial in $\boldsymbol{V}$ in order defined by columns }

\Input
\Desc{$\boldsymbol{V}$}{matrix of $\tau$ polynomials}
\Desc{$d$}{matrix of integers with the same dimensions as $\boldsymbol{V}$}
\EndInput
\Output
\Desc{$\boldsymbol{A}$}{column matrix of coefficients}
\EndOutput

\State $n \gets \Call{Rows}{\boldsymbol{V}}$
\State $m \gets \Call{Columns}{\boldsymbol{V}}$
\State $s \gets 1$

\For {$i = 1..n$}
    \For {$j = 1..m$}
        \For {$k = 1..d[i,j]-1$} 
            \State $\boldsymbol{A}[s, 1] \gets
            \Call{Coefficient}{{\boldsymbol{V}[i,j]}, {k}} $
            \State $s \gets s+1$
        \EndFor    
    \EndFor
\EndFor
\State \Return $\boldsymbol{A}$
\EndFunction
\end{algorithmic}
\end{algorithm}

\begin{algorithm}[H]
\caption{Computing dual of triangular \tm module}\label{alg:dual}
\begin{algorithmic}[1]

\Function{Dual}{$\Upsilon$}

\Input
\Desc{$\Upsilon$}{triangular \tm module without nilpotency and subquotients of rank $\geq2$}
\EndInput

\Output
\Desc{$\Pi$}{dual of $\Upsilon$}
\EndOutput

\State $n \gets \Call{Dim}{\Upsilon}$
\State $d$ : $n \times 1$ matrix

\For {$i = 1..n$}
\Comment{computing degrees of basis elements}
\State $d[i,1] \gets \Call{Deg}{\Upsilon[j,j]}$
\EndFor

\State $s \gets 1$
\For {$i = 1..n$}
\Comment{computing columns of $\Pi$}
        \State $r \gets \Call{Deg}{\Upsilon[j,j]}$
\For {$k = 2..r$}
            \State $\boldsymbol{V_{s}}
                \gets \Call{TMult}{C, {E_{i\times 1} c {\tau}^{k-1}}}$

       \algstore{myalg2}
\end{algorithmic}
\end{algorithm}
\begin{algorithm}
\begin{algorithmic}
\algrestore{myalg2}  
        
            \State $\boldsymbol{V_{s}} \gets
                \Call{Reduce2}
                {{\boldsymbol{V_{s}}}, {\Upsilon}, C}
            $
            \State $\boldsymbol{V_{s}} \gets
                \Call{CoefficientFormDual}{{\boldsymbol{V_{s}}}, {d}}$
            \State $\Call{Substitute}{\boldsymbol{V_{s}}, c^{(\_)} \to {\tau}^{(\_)}}$
            \State $s \gets s + 1$                 
        \EndFor
\EndFor
\State $\Pi \gets \begin{bmatrix}
    \boldsymbol{V_{1}} & ... & \boldsymbol{V_{s-1}} 
    \end{bmatrix}$
\State \Return $\Pi$

\EndFunction
\end{algorithmic}
\end{algorithm}

We now present the pseudocode for computing the double dual biderivation. Before doing so, let us recall the following facts.
\begin{enumerate}
\item By \cite[Lemma 3.2]{gkk}, the $\Fq[t]$–module $\Extt(\Upsilon^{\vee},C)$ inherits a natural \tm module structure arising from the \tm reduction algorithm.
\end{enumerate}

\begin{enumerate}
    \item By \cite[Lemma~3.2]{gkk}, the $\Fq[t]$–module $\Extt(\Upsilon^{\vee},C)$
inherits a natural \tm-module structure arising from the \tm reduction algorithm\footnote{%
Although the proof of Lemma~3.2 in \cite{gkk} is given for a lower-triangular block matrix, the argument carries over to our setting with only minor modifications.}.
\item By \cite[Lemma~9.2]{kk25}, one has an isomorphism
\begin{align*}\label{eq:doubledual-shape}
    {\Upsilon^\vee}^\vee \;&\cong\;
    \Big[
        \underbrace{\,0,\dots,0,\ K\tau\,}_{\rk\psi_2-1\ \textnormal{terms}}
        \;\Big|\;
        \underbrace{\,0,\dots,0,\ K\tau\,}_{\rk\psi_1-1\ \textnormal{terms}}
    \Big].   
\end{align*}
\item The matrix of $\delta^\vee$ has nonzero entries only in its last column; see \cite[Corollary 8.12]{kk25}.
\end{enumerate}
In addition, our algorithm makes use of the following inner biderivations from $\Derin(\psi,C)$, where $\psi=\theta+\sum_{i=1}^{r} a_i\tau^i$ is a Drinfeld module. To describe them conveniently, we introduce the following notation for row vectors in $\Mat_{1\times (r-1)}(K)$:
\begin{align*}
    V_{r}(\mu)
      &:= \big[\,
           \mu,\ \mu^{(1)},\ \dots,\ \mu^{(r-2)}
         \,\big], \\[0.2cm]
    W_{l,r}(\mu)
      &:= \big[
           \underbrace{0,\dots,0}_{l\ \textnormal{terms}},
           \ \mu,\ \mu^{(1)},\dots,\mu^{(r-l-1)}
         \big],
         \qquad l<r-1.
         \nonumber
\end{align*}
The matrices $V_r(\mu)$ and $W_{l,r}(\mu)$ give rise to the following two families of inner biderivations:
\begin{equation}\label{eq:biderywacje_wewnetrzne_cykl}
    \delta_{\psi^\vee,C}^{\big( V_{r}(\mu)\tau^k\big)}=
    \Big(\theta^{(k)}-\theta\Big)V_r(\mu)\tau^k+
    \Bigg[ 0,\cdots, 0, - \sum\limits_{j=1}^{r}\dfrac {a_{j}^{(k)}}{a_{r}^{(k)} } \mu^{(j-1)}\cdot\tau^{k+1} +  \dfrac{\mu}{a_{r}^{(k+1)} }\tau^{k+2} \Bigg]
    \end{equation}
    \begin{equation*}
     \delta_{\psi^\vee,C}^{\big( W_{l,r}(\mu)\tau^{k}\big)}=
      \Big(\theta^{(k)}-\theta\Big)W_{l,r}(\mu)\tau^k+
    \Bigg[ \podwzorem{0,\cdots,0}{(l-1)\textnormal{ terms}}, \mu\tau^{k+1},0,\cdots,0, - \sum\limits_{j =l+1}^{r}\dfrac {a_{j}^{(k)}}{a_{r}^{(k)} } \mu^{(j-l-1)}\tau^{k+1}\Bigg] \nonumber
\end{equation*}

\begin{algorithm}[H]\label{alg:doubledual}
\caption{Computing double dual biderivation of a triangular \tm module}
\begin{algorithmic}[1]

\Function{DoubleDualBiderivation}{$\Upsilon^\vee,\,r_1,\, r_2$}
\Input
\Desc{$\Upsilon^\vee$}{dual of \tm module $\Upsilon$}
\Desc{$r_1$}{rank of $\psi_1$}
\Desc{$r_2$}{rank of $\psi_2$}
\EndInput

\Output
\Desc{${\delta^\vee}^\vee$}{ double dual biderivation $(\delta^\vee)^\vee$ of $\Upsilon$}
\EndOutput

        \algstore{myalgdoubledual}
\end{algorithmic}
\end{algorithm}
\begin{algorithm}
\begin{algorithmic}
\algrestore{myalgdoubledual}

\State $v \gets \Call{TMult}{C,\; E_{1\times r_2-1}\, c\,\tau}$
\State $\mu\gets \Call{Coefficient}{v,2}\cdot \Call{Coefficient}{\psi_2,r_2}$
\State $G\gets \Call{TMult}{ \big[V_{r_2}(\mu)\mid \boldsymbol{0}\big],\Upsilon^\vee} - \Call{TMult}{C,\big[V_{r_2}(\mu)\mid \boldsymbol{0}\big]}$

\State $v\gets v - G$
\State $\mathrm{last}\gets r_1+r_2-2$
\For{$d=\Call{Deg}{v[\mathrm{last}]}, .., 2 $}
\State $\mu\gets \Call{Coefficient}{v[\mathrm{last}], d}\cdot  \Call{Coefficient}{\psi_1, r_1}^{(d-1)}$
\State $G\gets \Call{TMult}{ \big[\boldsymbol{0}\mid V_{r_1}(\mu)\tau^{d-2}\big],\Upsilon^\vee} - \Call{TMult}{C,\big[\boldsymbol{0} \mid V_{r_1}(\mu)\tau^{d-2}\big]}$ 
\State $v\gets v - G$
    \If{d>2}
 \For{$l=r_2, r_2+1,..,\mathrm{last}-1$}

            \State $\mu\gets \Call{Coefficient}{v[l],d-2}$
            \State $G\gets \Call{TMult}{ \big[\boldsymbol{0}\mid W_{l,r_1}(\mu)\tau^{d-3}\big],\Upsilon^\vee} - \Call{TMult}{C,\big[\boldsymbol{0} \mid W_{l,r_1}(\mu)\tau^{d-3}\big]}$ 
               \State $v\gets v - G$       
 
        \EndFor
    \EndIf
\EndFor

\State ${\delta^\vee}^\vee \gets \Call{Coefficient}{v[last],1}$

\State $\Call{Substitute}{ {\delta^\vee}^\vee,\; c^{(\_)} \to {\tau}^{(\_)} }$
\State \Return ${\delta^\vee}^\vee$
\EndFunction

\end{algorithmic}
\end{algorithm}

\begin{remark}
Let us note that, in the computation of the double dual biderivation, we encounter very large expressions $\mu$ (in lines 14 and 19 of the preceding algorithm), which are used to construct inner biderivations of the form \eqref{eq:biderywacje_wewnetrzne_cykl}. This requires particular care, since otherwise the computation of even a single such inner biderivation may become prohibitively time-consuming.
To improve efficiency, we adopted the following strategy. First, we determined the expression $G$ in lines 15 and 20 for an indeterminate $\mu$. We then grouped the resulting terms according to the powers $\mu^{(k)}.$ Finally, using the recurrence relation 
$
\mu^{(k+1)} = \bigl(\mu^{(k)}\bigr)^q
$
together with the additivity of the $q$-th power map in the field $K$, we computed these powers for the specific expression under consideration. This approach significantly accelerated the entire reduction process.
    \end{remark}

Finally, we present an algorithm for determining whether two biderivations $\delta_1$ and $\delta_2$ give rise to isomorphic triangular \tm modules
$
\Upsilon_1=\Upsilon(\psi_1,\psi_2,\delta_1)
\quad\text{and}\quad
\Upsilon_2=\Upsilon(\psi_1,\psi_2,\delta_2),
$
under the assumption that $\operatorname{rk}(\psi_1)>\operatorname{rk}(\psi_2)$.
The algorithm relies on the fact that $\Upsilon_1\cong\Upsilon_2$ if and only if there exists a skew polynomial $$u(\tau)\in K\{\tau\}$$ such that
$
\delta_1-\delta_2={\delta}^{u(\tau)}_{\psi_2,\psi_1}\in \Derin(\psi_2,\psi_1).
$
Moreover, whenever such a polynomial exists, the algorithm returns the corresponding polynomial $u(\tau)$.

\begin{algorithm}[H]
\caption{Computing polynomial $u(\tau)$ that induces an isomorphism   $\Upsilon_1=\Upsilon(\psi_1,\psi_2,\delta_1)\cong
\Upsilon_2=\Upsilon(\psi_1,\psi_2,\delta_2)$}

\begin{algorithmic}[1]
\Function{ComputeIzo}{$\psi_1,\psi_2,\delta_1,\delta_2$}
\Input
\Desc{$\psi_1,\psi_2$}{\quad Drinfeld modules}
\Desc{$\delta_1$}{\quad biderivation of $\Upsilon_1$}
\Desc{$\delta_2$}{\quad biderivation of $\Upsilon_2$}
\EndInput
\Output
 \,\,  a skew polynomial $u(\tau)$ such that $u(\tau)$ induces an isomorphism $\Upsilon_1 \cong \Upsilon_2$, if one exists; otherwise, $\mathbf{false}$.
\EndOutput
\State $r\gets \rk\psi_1$
\State $\delta\gets \delta_1-\delta_2$
\State $u\gets 0$

        \algstore{myalg4}
\end{algorithmic}
\end{algorithm}
\begin{algorithm}
\begin{algorithmic}
\algrestore{myalg4}   

\While{$\Call{Deg}{\delta}\geq r$}
    \State $d\gets \Call{Deg}{\delta}$
  
    \State $\mu\gets -\dfrac{\Call{Coefficient}{\delta, d}}{ \Call{Coefficient}{\psi_1,\rk\psi_1  } }$

    \State $G\gets \Call{pMult}{\mu \tau^{d-r},\psi_2 }-\Call{pMult}{\psi_1,\mu\tau^{d-r}}$
    \State $u\gets u + \mu \tau^{d-r}$
    \State $\delta\gets \delta - G$
\EndWhile
\If{$\delta = 0$}
    \State \Return $u$
\Else
    \State \Return $\boldsymbol{false}$
\EndIf   
\EndFunction
\end{algorithmic}
\end{algorithm}
All algorithms described in this chapter were implemented in Wolfram Mathematica 13.1.

\section{Symbolic complexity of dual and double dual computations}\label{symb}

The explicit computation of duals and double duals of \tm modules requires working with concrete elements of the twisted polynomial ring $K\{\tau\}$. Although the algebraic structure of these objects is well understood at an abstract level, their explicit representations often involve symbolic expressions whose size and complexity grow rapidly during computation. This phenomenon is closely related to the well-known problem of expression swell in symbolic computation; see \cite{vdH22}.
In the context of Drinfeld modules and \tm modules, similar computational challenges have been observed in explicit calculations over $K\{\tau\}$; see \cite{pr,gkk,kk25}.  These works indicate that, even when degrees and ranks can be controlled theoretically, the resulting coefficients may become prohibitively large from a computational perspective.
In this section, we investigate this phenomenon experimentally for triangular \tm modules of dimension two. We first study the dual construction and then the double dual, demonstrating that symbolic complexity increases significantly at each stage. Our experiments show that computer-assisted methods become indispensable even for examples of moderate size.

\subsection{Growth of symbolic expressions in the dual \tm module and its mitigation.}
Let       $\Upsilon=\Upsilon(\psi_1,\psi_2,\delta_1)$ be a triangular \tm-module of
dimension two without nilpotence. 
In this case, the dual biderivation $\delta^\vee$ is represented by a matrix
whose only nonzero entries occur in the last column.

From a theoretical perspective, the $\tau$-degree of $\delta^\vee$ can be
controlled explicitly. In particular, we obtain the following result.

\begin{proposition}\label{prop:deg_dual_module}
    Let $\Upsilon=\Upsilon(\psi_1,\psi_2,\delta)$ be a triangular \tm-module of
    dimension $2$ without nilpotence, where $\rk\psi_1>\rk\psi_2>1$. Then
    \[
        \deg_\tau\Upsilon^\vee
        =
        \max\left\{
            2,\;
            \left\lfloor
            \frac{\deg_\tau\delta-1}{\rk\psi_2-1}
            \right\rfloor+1
        \right\}.
    \]
\end{proposition}
Thus, the $\tau$-degree of the dual \tm module is determined solely by the rank of $\psi_2$	
 and the $\tau$-degree of the biderivation $\delta$. In particular, its growth is linear in $\deg_{\tau}{\delta} ,$  providing an explicit theoretical bound despite the potentially large symbolic expressions arising in the computation of $\Upsilon^{\vee}.$ 

  The proof of Proposition \ref{prop:deg_dual_module}5 is given in Section~\ref{sec:explicit-dual}. Although the proposition provides explicit control over the $\tau$-degree of the dual biderivation, this control does not reflect the full computational complexity of the dual \tm\ module. In practice, the coefficients of ${\delta}^{\vee}$ 
quickly grow into large algebraic expressions, leading to significant expression swell.

To quantify this effect, we measure the symbolic size of ${\delta}^{\vee}$
  using two standard indicators: the {\bf leaf count}, which records the number of nodes in the corresponding expression trees, and the {\bf memory footprint}, measured by the number of bytes required to store the resulting expressions. Since ${\delta}^{\vee}$ 
has only one nonzero column (see \cite[Corollary 8.12]{kk25}), both quantities are obtained by summing the corresponding values over the entries of that column.
Table~1 illustrates the rapid growth of symbolic complexity for a collection of examples with increasing input degrees. Even for moderately sized parameters, the resulting expressions become too large for practical hand computations, highlighting the necessity of computer-assisted methods.
\begin{table}[htbp]
    \begin{tabularx}{\textwidth}{c@{\hspace{1.2em}}c@{\hspace{1.2em}}c@{\hspace{1.2em}}c}
    \label{tab:dual-growth}\\
\toprule
$(\rk\psi_1,\rk\psi_2,\deg\delta_1)$ &
$\deg_\tau(\delta^\vee)$ &
LeafCount$(\delta^\vee)$ &
ByteCount$(\delta^\vee)$ \\
\midrule
(3,2,2) & 2 & 77 & 2600 \\   
(4,2,3) & 3 & 224 & 7560 \\
(5,2,4) & 4 & 668 & 22312 \\
(6,2,5) & 5 & 2345 & 77176 \\
(7,2,6) & 6 & 7664 & 249696 \\
(8,2,7) & 7 & 24972 & 807968 \\
(9,2,8) & 8 & 78210 & 2517136 \\
(10,2,9) & 9 & 243268 & 7797176 \\
\midrule
(4,3,3) & 2 & 110 & 3752 \\
(5,3,4) & 2 & 266 & 9048 \\
(6,3,5) & 3 & 687 & 23232 \\
(7,3,6) & 3 & 1788 & 59856 \\
(8,3,7) & 4 & 4833 & 159968 \\
(9,3,8) & 4 & 12751 & 418248 \\
(10,3,9) & 5 & 32877 & 1070504 \\
(11,3,10) & 5 & 83943 & 2717088 \\
(12,3,11) & 6 & 211546 & 6813872 \\
\bottomrule
\end{tabularx}
\caption{Growth of symbolic complexity in the computation of  $\delta^\vee$.}
\end{table}

To mitigate the growth of symbolic complexity, we make use of two algebraic observations. First, \cite[Corollary 8.8]{kk25} shows that the map $\delta\mapsto\delta^\vee$ is  $\Fq$-linear. It therefore suffices to carry out both the computations and the proofs for biderivations of the form $d\tau^h$. 

Next, suppose that the leading coefficients of 
$\psi_1$ and $\psi_2$	
  are not equal to $1.$ Then the coefficients of $\delta^\vee$
  are generally rational expressions whose denominators contain powers of these leading coefficients. The following proposition shows that, without loss of generality, one may reduce to the case in which the leading coefficients are equal to $1.$

\begin{proposition}\label{prop:monic_drinfeld}
Let $\Upsilon=\Upsilon(\psi_1,\psi_2,\delta)$ be a triangular \tm-module defined
over a field $K$. Then there exists a finite field extension $K(\Upsilon)$ of $K$
such that both $\psi_1$ and $\psi_2$ become monic.
\end{proposition}
\begin{proof}
Let $a_{r_1}$ and $b_{r_2}$ denote the leading coefficients of $\psi_1$ and
$\psi_2$, respectively, where $r_i=\rk\psi_i$ for $i=1,2$. Let $\mu$ and $\xi$
be roots of the polynomials
$
x^{q^{r_1}}-a_{r_1}x
\,\,\text{and}\,\,
x^{q^{r_2}}-b_{r_2}x,
$
respectively. Define
$
D=\operatorname{diag}(\mu,\xi).
$
Then the \tm-module $D\Upsilon D^{-1}$ is defined over
$
K(\Upsilon):=K(\mu,\xi),
$
and its diagonal entries are monic.
\end{proof}
If both $\psi_1$ and $\psi_2$ are monic, then the triangular \tm module
$\Upsilon=\Upsilon(\psi_1,\psi_2,\delta)$ will be called \emph{monic}.
Consequently, after a finite extension of the base field $K$, we may assume
without loss of generality that $\Upsilon$ is monic.

We emphasize that this assumption substantially simplifies the computations.
Indeed, in the monic case, the coefficients of the dual $t$-module are
polynomial expressions, whereas in general they are rational expressions
whose denominators involve powers of the leading coefficients of $\psi_1$
and $\psi_2$.

\subsection{Symbolic expression growth in the double dual and its mitigation}
We now turn to the double dual construction. In contrast to the dual case, the
double dual ${\delta^\vee}^\vee$ is represented by a single twisted polynomial,
that is, by a $1\times 1$ matrix. Consequently, its symbolic complexity can be
measured directly on this polynomial using the same indicators as above.

Despite this apparent simplification in form, the symbolic size of
${\delta^\vee}^\vee$ typically grows even more rapidly. Although its
$\tau$-degree remains controlled, its coefficients involve substantially larger
algebraic sub-expressions, reflecting the cumulative effect of successive dual
constructions in $K\{\tau\}$.

Table~\ref{tab:double-dual-growth} summarises this behaviour for the family
of examples, with the additional assumption that all triangular \tm-modules
under consideration are monic. The observed growth clearly demonstrates that
the computation of double dual modules lies well beyond the scope of feasible
manual calculations and further highlights the necessity of computer-assisted
experimentation.
\begin{table}[htbp]
    \begin{tabularx}{\textwidth}{c@{\hspace{1.2em}}c@{\hspace{1.2em}}c@{\hspace{1.2em}}c}
    \label{tab:double-dual-growth}\\
\toprule
$(\rk\psi_1,\rk\psi_2,\deg\delta_1)$ &
$\deg_\tau\big({\delta^\vee}^\vee\big)$ &
LeafCount$\big({\delta^\vee}^\vee\big)$ &
ByteCount$\big({\delta^\vee}^\vee\big)$ \\
\midrule
(3,2,2) & 3 & 41 & 1408 \\
(4,2,3) & 7 & 476 & 16224 \\
(5,2,4) & 13 & 6749 & 226544 \\
(6,2,5) & 21 & 114389 & 3810120 \\
(7,2,6) & 31 & 2167524 & 71894768 \\
(8,2,7) & 43 & 46453246 & 1536367808 \\
\midrule
(4,3,3) & 4 & 63 & 2192 \\
(5,3,4) & 6 & 237 & 8192 \\
(6,3,5) & 11 & 1460 & 49856 \\
(7,3,6) & 14 & 6634 & 224664 \\
(8,3,7) & 22 & 41708 & 1404768 \\
(9,3,8) & 26 & 221505 & 7421632 \\
(10,3,9) & 37 & 1429731 & 47796176\\
(11,3,10) & 42 & 8306737 & 276785784 \\
\bottomrule
\end{tabularx}
\caption{Growth of symbolic complexity in the double dual ${\delta^\vee}^\vee$.}
\end{table}

To mitigate symbolic expression growth in the computation of the double dual
$t$-module, we make use of the fact that the map
$\delta^\vee \longmapsto {\delta^\vee}^\vee$
is $\mathbb{F}_q$-linear; see \cite[Corollary~8.8]{kk25}. It therefore suffices
to carry out our computations and proofs for biderivations of the form
$d\tau^h$.

In addition, we represent the dual biderivation as the sum
\eqref{eq:postac_duala} and perform the computations separately for each of its
summands. This substantially reduces the growth of intermediate symbolic
expressions.

\section{The polynomial $u(\tau)$ defining the  isomorphism    ${\Upsilon^\vee}^\vee\cong \Upsilon$. }\label{poly}
In this chapter, we present the conclusions drawn from the computations concerning the form of the polynomial $u(\tau)$ that determines the isomorphism between ${\Upsilon^\vee}^\vee$ and $\Upsilon$. Throughout our computations, we assume the following explicit forms of $\psi_1$, $\psi_2$, and $\delta$:
\[
\psi_1=\theta+\sum_{i=1}^{r_1-1} a_i\tau^i+\tau^{r_1}, \qquad
\psi_2=\theta+\sum_{j=1}^{r_2-1} b_j\tau^j+\tau^{r_2}, \qquad
\textnormal{and} \qquad
\delta=d\tau^h.
\]
Each example is indexed by a triple  $(r_1,r_2,h)$ satisfying $r_1>h\geq r_2.$ We denote by $u(r_1,r_2,h)$ the polynomial $u(\tau)$ associated with the triple $(r_1,r_2,h)$.

 We fix $r_2\in\mathbb{N}$ and organize the polynomials $u(r_1,r_2,h),$ with $r_1>h\geq r_2,$  into classes indexed by pairs $(r_1,r_2)$ satisfying $r_1>r_2.$ For each such pair, we define
\[
\mathcal{T}_{r_1,r_2}
=
\Big\{u(r_1,r_2,h):\ h\in\{r_2,r_2+1,\dots, r_1-1\} \Big\}.
\]
For a fixed $r_2,$ these classes are considered successively in increasing order of $r_1.$ By comparing the elements of two consecutive classes $\mathcal{T}_{r_1,r_2}$ and $\mathcal{T}_{r_1+1,r_2},$ 	
 we observed the following relationships:
\begin{itemize}
    \item[\textbf{a.}] In two consecutive classes $\mathcal{T}_{r_1,r_2}$ and $\mathcal{T}_{r_1+1,r_2}$, the first $r_2-1$ polynomials coincide, i.e.,
    \[
    u(r_1,r_2,h)=u(r_1+1,r_2,h),
    \qquad
    h=r_2,r_2+1,\ldots,2(r_2-1).
    \]

    \item[\textbf{b.}] For $h=2r_2-1,2r_2,\ldots,3(r_2-1)$, the polynomial
    $u(r_1+1,r_2,h)$ in the class $\mathcal{T}_{r_1+1,r_2}$ depends
    recursively on its counterpart $u(r_1,r_2,h)$ in the class
    $\mathcal{T}_{r_1,r_2}$,

    \item[\textbf{c.}]
    The last polynomial $u(r_1,r_2,r_1)$  in the class $\mathcal{T}_{r_1,r_2},$	
  as well as the polynomials $u(r_1,r_2,h)$ for $h>3(r_2-1),$ are new in the sense that they are independent of all polynomials appearing in the preceding classes.
   \item[\textbf{d.}]
  The polynomials $u(r_1,r_2,h)$ for $h>3(r_2-1)$  have a significantly more complicated structure than those corresponding to the case $h\leq 3(r_2-1)$.
\end{itemize}

The computations summarized in Table~\ref{tab:postac_dla_u}, together with the patterns observed in the families of polynomials $u(r_1,r_2,h)$, provide strong evidence for the validity of Cartier--Nishi duality in the two-dimensional triangular case. Based on these computations and observations, we formulate the following conjecture.

\begin{conjecture}
Let $\Upsilon=\Upsilon\big(\psi_1,\psi_2, \delta\big)$ be a triangular \tm module without nilpotence, where $\rk\psi_1>\rk\psi_2>1$. Then
 the Cartier--Nishi theorem holds.
 \end{conjecture}

Below we present a sample table describing the polynomials $u(r_1+1,r_2,h)$ for the classes$\mathcal{T}_{4,3}$, $\mathcal{T}_{5,3}$, $\mathcal{T}_{6,3}$, $\mathcal{T}_{7,3}$, and part of the class $\mathcal{T}_{8,3}$ satisfying $h\leq 3(r_2-1)$.

\begin{table}[htbp]
    \begin{tabularx}{\textwidth}{ccc@{\hspace{1.2em}}Y}
\toprule
$r_1$ & $r_2$ & $h$ & $u(r_1,r_2,h)$ \\
\midrule

$4$ & $3$ & $3$ & $-d$ \\[0.6em]
\midrule

$5$ & $3$ & $3$ & $u(4,3,3)$ \\[0.6em]
$5$ & $3$ & $4$ & $d b_2^{(1)} - d \tau$ \\[0.6em]
\midrule 
$6$ & $3$ & $3$ & $u(4,3,3)$ \\[0.6em]
$6$ & $3$ & $4$ & $u(5,3,4)$  \\[0.6em]
$6$ & $3$ & $5$ &
$ b_1^{(2)} d - b_2^{(1)+(2)} d - 
 a_1 d^{(1)} + \Big(b_2^{(2)} d - a_2 d^{(2)}\Big) \tau + \Big(-d - a_3 d^{(3)}\Big) \tau^2 - 
 a_4 d^{(4)} \tau^3 $
\end{tabularx}
\end{table}

\begin{table}[htbp]
\begin{tabularx}{\textwidth}{ccc@{\hspace{1.2em}}Y}
 &  & &
$  - a_5 d^{(5)} \tau^4 - d^{(6)} \tau^5$\\[0.6em]
 \midrule
$7$ & $3$ & $3$ & $u(4,3,3)$ \\[0.6em]
 $7$ & $3$ & $4$ & $u(5,3,4)$  \\[0.6em]
$7$ & $3$ & $5$ & $u(6,3,5) + (1 - a_6) d^{(6)} \tau^5-d^{(7)} \tau^6$  \\[0.6em] 
\midrule
$7$ & $3$ & $6$ & 
  $d \Big(\theta -\theta ^{(3)} + b_1^{(3)} b_2^{(1)}+ b_1^{(2)} b_2^{(3)}-
   b_2^{(3)+(2)+(1)}\Big)-a_1 d^{(1)}\Big( b_2^{(4)} +a_1 b_2^{(1)}\Big) 
   +  \left(a_1 d^{(1)}-d b_1^{(3)}+d b_2^{(3)+(2)}-a_2 d^{(2)}\Big( b_2^{(2)} +a_2 b_2^{(5)} \Big)\right)\tau
   +\Big(a_2 d^{(2)}-d b_2^{(3)} 
   -a_3 d^{(3)}\Big(b_2^{(3)} +b_2^{(6)} \Big) \Big)\tau ^2 
   + \left(d+a_3 d^{(3)}-a_4 d^{(4)}\Big( b_2^{(4)} + b_2^{(7)} \Big)\right)\tau ^3
    + \left(a_4 d^{(4)}-a_5 d^{(5)}\Big(b_2^{(5)} +b_2^{(8)}\Big)\right) \tau ^4
    +\left(a_5 d^{(5)} -a_6d^{(6)} \Big(b_2^{(6)}+b_2^{(9)} \Big)\right)\tau ^5
   +\left(a_6 d^{(6)}  -d^{(7)}\Big(b_2^{(7)} +b_2^{(10)}\Big) \right)\tau ^6
   +\tau ^7 d^{(7)}$
   \\ 
\midrule
   $8$ & $3$ & $3$ & $u(4,3,3)$  \\[0.6em]
$8$ & $3$ & $4$ & $u(5,3,4)$\\[0.6em]
$8$ & $3$ & $5$ & $u(7, 3, 5) + \big(1 - a_7\big) d^{(7)} \tau^6 - d^{(8)} \tau^7$ \\[0.6em]
$8$ & $3$ & $6$ & $u(7, 3, 6)+d^{(7)}\Big(1-a_7\Big)\Big( b_2^{(7)}+ b_2^{(10)} \Big)\tau^6-\Big( d^{(7)}\big(1-a_7\big)  +d^{(8)}\Big( b_2^{(8)}-b_2^{(11)}\Big)  \Big)\tau^7+d^{(8)}\tau^8$  \\[0.6em]
\bottomrule
\end{tabularx}
\caption{The polynomial $u(\tau)$ inducing the isomorphism
$\Upsilon \cong (\Upsilon^\vee)^\vee$ for selected values of $(r_1,r_2,h)$ for $h\leq 3(r_2-1)$.}
\label{tab:postac_dla_u}
\end{table}
In  Table \ref{tab:postac_dla_u_837} we provide several examples for which the condition $h>3(r_2-1)$ is satisfied.
\begin{table}[htbp]
\begin{tabularx}{\textwidth}{ccc@{\hspace{1.2em}}c@{\hspace{1.2em}}c}
\toprule
$r_1$ & $r_2$ & $h$ & 
LeafCount$\big(u(r_1,r_2,h)\big)$ &
ByteCount$\big(u(r_1,r_2,h)\big)$  \\
\midrule
$8$ & $3$ & $7$ & $1634$ & $54776$\\
\midrule
$9$ & $3$ & $7$ & $1902$ & $63944$\\
$9$ & $3$ & $8$ & $6221$ & $209224$\\
\midrule
$10$ & $3$ & $7$ & $2188$ & $73752$\\
$10$ & $3$ & $8$ & $7236$ & $243896$\\
$10$ & $3$ & $9$ & $40141$ & $1354216$\\
\bottomrule
\end{tabularx}
\caption{Growth of symbolic complexity in the polynomial $u(r_1,r_2,h)$ for $h>3(r_2-1)$.}
\label{tab:postac_dla_u_837}
\end{table}

From the computations performed, we are naturally led to the conclusion that there is a substantial difference between the structures of the polynomials $u(r_1,r_2,h)$ for $h\leq 3(r_2-1)$ and for $h>3(r_2-1).$ The former exhibits a distinctly recursive and repetitive pattern, whereas the latter does not appear to possess these characteristics. In the present work, we shall focus exclusively on the study of the former structure. To this end, we introduce the following definition.

\begin{definition}\label{def:ALD}
Let $\Upsilon=\Upsilon(\psi_1,\psi_2,\delta)$ be a triangular module with reduced form
$\Upsilon^{\mathrm{red}}
=\Upsilon(\psi_1,\psi_2,\delta^{\mathrm{red}})$
We say that $\Upsilon$ allows for \emph{Almost Low Degree biderivations} (abbreviated as $\mathrm{ALD}$-biderivations) if
$\deg_{\tau}\delta^{\mathrm{red}}
\leq
3(\operatorname{rk}\psi_2-1).$
\end{definition}

Recall that the reduced form of a \tm module $\Upsilon$ need not be defined over the base field $K$, but may instead be defined over a finite field extension $K(\Upsilon)$ of $K$.

In this work, we shall prove the following theorem.
\begin{theorem}\label{thm:main}
 Let $\Upsilon=\Upsilon\big(\psi_1,\psi_2, \delta\big)$ be a triangular \tm module without nilpotence, where $\rk\psi_1>\rk\psi_2>1$. If $\Upsilon$ allows for $\mathrm{ALD}$-biderivations, then
 the Cartier--Nishi theorem holds.
\end{theorem}

\section{Explicit form of the dual \tm module}\label{sec:explicit-dual}

We begin this section by proving Proposition \ref{prop:deg_dual_module}. The techniques developed in the proof will subsequently be used to determine the explicit form of the dual \tm module.

\begin{proof}[Proof of Proposition \ref{prop:deg_dual_module}]
It is enough to consider generators of the form $d\tau^h$; see \cite[Corollary 8.8]{kk25}. If $h<r_2$, then \cite[Proposition 8.11]{kk25} implies that $\deg_\tau \Upsilon^\vee = 2$, and hence the assertion of the proposition follows.

Now suppose that $h\ge r_2$. By \cite[Lemma 8.9]{kk25}, there is an isomorphism of ${\mathbb F}_q$-linear spaces
$$\Upsilon^\vee\cong  V_2=\Big[ K\tau+K\tau^2+\cdots+K\tau^{r_2-1}\, ,\, K\tau+K\tau^2+\cdots+K\tau^{r_1-1} \Big].$$
   In order to determine the \tm module structure  on $\Upsilon^\vee$ we transfer the multiplication    \eqref{eq:wzor_na_mnozenie_w_Ext} via the above  isomorphism. Consider the following coordinate system on  $V_2$:
    
    \begin{equation}
        \label{eq:baza:dual}
        \Big(\big[ \tau^k,0 \big] \Big)_{k=1}^{r_2-1}, \Big( \big[ 0,\tau^k\big] \Big)_{k=1}^{r_1-1}.
    \end{equation}
    Next, we compute the results of the action of $t$ on the generators  $\big[ c\tau^k,0 \big]$ and  $\big[ 0,c\tau^k\big]$ and express the obtained results in the chosen coordinate system.  

    Since we are only interested in determining $\delta^\vee$, it suffices to compute the action of $t$ on those generators that contribute to the form of $\delta^\vee$.

    If  $k\in\{1,\dots, r_1-2\}$, then
    $$t*\big[0,c\tau^k\big]= \Big[ 0,(\theta+\tau)c\tau^k \Big]  =  \Big[ 0, \theta c\tau^k +c^{(1)} \tau^{k+1}  \Big]\in V_2,$$
   then  
    $t*\big[ 0, c\tau^k\big]$ has the following representation in the fixed basis:
    $$t*\big[0,c\tau^k\big]=\Big[\ 0,\dots,0\ \Big|\ \podwzorem{0,\dots, 0}{k-1 \textnormal{ terms}}, \theta , \tau , 0, \dots, 0\ \Big].$$
    This shows that all entries of the matrix $\delta^\vee$ outside the last column are zero.

    Consider now the last column of
    $\delta^\vee$. The terms there come from the action of  $t$ on the generator $\big[ 0, c\tau^{r_1-1}\big]$. Then
    \begin{align*}
        t*\big[ 0, c\tau^{r_1-1}\big]&= (\theta+\tau)\big[ 0, c\tau^{r_1-1}\big]=\Big[ 0, \theta c\tau^{r_1-1} + c^{(1)}\tau^{r_1} \Big]\notin V_2. 
    \end{align*}
    Since the degree of $\theta c\tau^{r_1-1} + c^{(1)}\tau^{r_1}$ equals 
    $r_1$ we reduce it by means of the following inner biderivation
    $$\delta^{\big(\big[0,  \mu\tau^0 \big]\big)}_{\Upsilon,C}=\Big[\, \mu d\tau^h\, \big|\, \delta^{(\mu\tau^0)}_{\psi_1,C}     \,\Big ] \quad\textnormal{for}\quad 
    \mu =\dfrac{c^{(1)}}{a_{1, r_{1} }}.$$ 
    As a result one gets  the following equality:
    \begin{align*}
         t*\big[ 0, c\tau^{r_1-1}\big]&= \Bigg[\ -\dfrac{c^{(1)}}{a_{1, r_{1} }}\cdot  d \tau^{h}\ \Big|\  \textnormal{terms for } \psi_{1}^\vee\  \Bigg]\notin V_2.
    \end{align*}
    Our goal is to perform successive reductions of the first coordinate until its degree becomes less than $r_2$. To achieve this, we will use inner biderivations of the form
  $$\delta^{\big(\big[ \mu\tau^l,0 \big]\big)}_{\Upsilon,C}=\Big[\, \delta^{(\mu\tau^l)}_{\psi_2,C}\, \big|\, 0 \,\Big ] \quad\textnormal{where}\quad 
  \delta^{(\mu\tau^l)}_{\psi_2,C}=\mu\Big(\theta^{(l)}-\theta\Big)\tau^l+ \Big(\mu  b_1^{(l)}-\mu^{(1)}\Big)\tau^{l+1}+\sum\limits_{j=2}^{r_2}\mu b_j^{(l)}\tau^{j+l}.
   $$ 
   Since the biderivation $\delta^{\big(\big[ \mu\tau^l,0 \big]\big)}_{\Upsilon,C}$ depends only on $l\in\mathbb{Z}_{\geq 0}$ and the element $\mu\in K$, we will abbreviate it, for simplicity, as $(l,\mu)$. 
   We start with the reduction of the term
   $-\dfrac{c^{(1)}}{a_{1, r_{1} }}\cdot  d \tau^{h}$ by means of the inner biderivation  $(h-r_2,c^{(1)}\xi_{1,h-r_2})$ for the suitably chosen $\xi_{1,h-r_2}\in K$. As a result of this reduction, terms involving $c^{(1)}$ of degrees from $h-r_2$ to $h-1$ may appear.\footnote{Depending on which coefficients $b_j$ in the biderivation are nonzero, terms involving $c^{(1)}$ may arise. This does not affect the argument, since the inner biderivation $(l,0)$ is zero. On the other hand, the term involving $c^{(2)}$ of degree $h-(r_2-1)$ is always present, regardless of the form of $\psi_2$.}

We now eliminate all terms involving $c^{(1)}$. This is accomplished by means of the following sequence of inner biderivations

\begin{equation}\label{eq:pierwszy_ciag_redukcji}
      \big( l, c^{(1)}\xi_{1,l}\big)\quad \textnormal{for}\quad l=h-r_2-1,h-r_2-2,\dots, 2,1,0,
  \end{equation}
  where $\xi_{1,l}$ are suitably chosen elements from the field
  $K$.\footnote{At the moment we are not interested in the explicit form of the elements $\xi_{1,l}$, although it is worth noting that they are described by linear recurrences.}
  After carrying out this sequence of reductions on the first coordinate, terms involving $c^{(2)}$ appear in degrees starting from $h-(r_2-1)$. These terms are eliminated by means of the following sequence of inner biderivations.
\begin{equation}
    \label{eq:drugi_ciag_redukcji}
    \big( l, c^{(2)}\xi_{2,l}\big)\quad \textnormal{for some}\quad \xi_{2,l}\in K\quad\textnormal{and}\quad l= h-2r_2+1,h-2r_2,\dots, 2,1,0.
\end{equation}
Similarly, these reductions give rise to terms involving $c^{(3)}$ beginning in degree
$l = h - 2(r_2-1).$
They are eliminated by applying the following sequence of inner biderivations.
\begin{equation}
    \label{eq:trzeci_ciag_redukcji}
    \big( l, c^{(3)}\xi_{3,l}\big)\quad \textnormal{for some}\quad \xi_{3,l}\in K\quad\textnormal{and}\quad l= h-3r_2+2,h-3r_2+1,\dots, 2,1,0.
\end{equation}
Continuing this process we use the following sequences of inner biderivations:
\begin{align}\label{eq:kolejne_ciagi_redukcji}
    \big( l, c^{(4)}\xi_{4,l}\big)&\quad \textnormal{for}\quad l= h-3(r_2-1)-r_2,  \dots, 1,0,\\
     \big( l, c^{(5)}\xi_{5,l}\big)&\quad \textnormal{for}\quad l= h-4(r_2-1)-r_2, \dots, 1,0,\nonumber\\  
     \vdots& \nonumber\\ 
      \big( l, c^{(q)}\xi_{q,l}\big)&\quad \textnormal{for}\quad l= h-(q-1)(r_2-1)-r_2, \dots, 1,0. \nonumber
\end{align}
where  $q$ is the least natural number satisfying
$$(\star)\quad h-q(r_2-1)<r_2.$$
Finally we obtain that
 \begin{align*}
         t*\big[ 0, c\tau^{r_1-1}\big]&= \Bigg[\ \sum\limits_{j=1}^{r_2-1}\sum\limits_{i=1}^{q+1}\alpha_{i,j}c^{(i)} \tau^j \Big|\  \textnormal{terms for } \psi_{1}^\vee\  \Bigg]\in V_2,
    \end{align*}
    for some $\alpha_{i,j}\in K$.
   Expressing this in the basis  \eqref{eq:baza:dual} we obtain 
    \begin{align} \label{eq:postac_duala_w_propozycji}
         t*\big[ 0, c\tau^{r_1-1}\big]&= \Bigg[\, \sum\limits_{i=1}^{q+1}\alpha_{i,1}c^{(i)},\, 
         \sum\limits_{i=1}^{q+1}\alpha_{i,2}c^{(i)},\, \dots,\, \sum\limits_{i=1}^{q+1}\alpha_{i,r_2-1}c^{(i)}\, \Big|\  \textnormal{terms for } \psi_{1}^\vee\  \Bigg]\\
         &= \Bigg[\, \sum\limits_{i=1}^{q+1}\alpha_{i,1}\tau^i,\, 
         \sum\limits_{i=1}^{q+1}\alpha_{i,2}\tau^i,\, \dots,\, \sum\limits_{i=1}^{q+1}\alpha_{i,r_2-1}\tau^i\, \Big|\  \textnormal{terms for } \psi_{1}^\vee\  \Bigg]_{\tau=c}.\nonumber
    \end{align}
    The condition $(\star)$ implies that $\deg_\tau\Upsilon^\vee=q+1=\left\lfloor
            \frac{h-1}{r_2-1}
            \right\rfloor+1$ and the proof is completed.  
\end{proof}
Determining the explicit forms of $\xi_{k,l}$ and $\alpha_{i,j}$ yields a concrete description  of the dual module $\Upsilon^\vee$. 
Before we carry this out in the next theorem,
consider the following example.
\begin{example}\label{exaple736}
    Let $\Upsilon=\Upsilon(\psi_1,\psi_2,\delta)$ be a triangular \tm module such that
    $$\psi_1=\theta+\sum\limits_{i=1}^6 a_i\tau^i+\tau^7,\quad \psi_2=\theta+\sum\limits_{j=1}^2 b_j\tau^j+\tau^3,
    \quad
    \delta=d\tau^6.$$
   Analogously as in the former proof, one can determine the dual \tm module
     $\Upsilon^\vee$. We obtain the following reccursive relations for $\xi_{k,l}$:
     \begin{align}
        \label{rekurencje_dla_mu:7:3:6}
        \xi_{1,3}&=-d,\qquad\xi_{1,2}=-\xi_{1,3}b_2^{(3)}, \qquad\xi_{1,1}= -\xi_{1,3}b_1^{(3)}-\xi_{1,2}b_2^{(2)},\\\nonumber
        \xi_{1,0}&= \xi_{1,3}\big( \theta-\theta^{(3)}\big)-\xi_{1,2}b_1^{(2)} -\xi_{1,1}b_2^{(1)}\\ \nonumber
        \xi_{2,1}&=\xi_{1,3}^{(1)},\qquad \xi_{2,0}=\xi_{1,2}^{(1)}-\xi_{2,1}b_2^{(1)}. 
    \end{align}
    Additionally, the dual biderivation  $\delta^\vee$ has the following form:
    \begin{align*}
        \delta^\vee=\left[\begin{array}{cccc}
             0&\cdots &0&   \Big( \xi_{1,1}\big(\theta -\theta^{(1)}\big)-\xi_{1,0}b_1\Big)\tau+ \Big(\xi_{1,0}^{(1)}-\xi_{2,0}b_1+\xi_{2,1}\big(\theta -\theta^{(1)}\big)\Big)\tau^2 +\xi_{2,0}^{(1)} \tau^3\\
             0&\cdots &0&  \Big(\xi_{1,2}\big(\theta -\theta^{(2)}\big)-\xi_{1,0}b_2-\xi_{1,1} b_1^{(1)}\Big)\tau+\Big(\xi_{11}^{(1)}-\xi_{2,0}b_2-\xi_{2,1} b_1^{(1)}\Big)\tau^2+\xi_{2,1}^{(1)}\tau^3
        \end{array}\right].
    \end{align*}
Notice that, by the form of \eqref{rekurencje_dla_mu:7:3:6}, both $\xi_{2,1}$ and $\xi_{2,2}$ are $q$-th powers. This observation is important because, in determining the reduced form of the double dual of this module, we will make use of the $q$-th roots $\xi_{2,1}^{(-1)}$ and $\xi_{2,2}^{(-1)}$, which belong to the field of definition $K$.

   \end{example}

Denote 
$$\Delta_{k,l}(\xi)=
\left[\begin{array}{c}
0\\
\vdots \\
0\\
\xi\Big(\theta - \theta^{(l)}\Big)\tau^k\\[3mm]
 \xi^{(1)}\tau^{k+1} -\xi b_{1}^{(l)}\tau^k\\[3mm]
  -\xi b_{2}^{(l)}\tau^k\\[3mm]
  -\xi b_{3}^{(l)}\tau^k\\ \vdots\\
  -\xi b_{r_2-1-l}^{(l)}\tau ^k     
\end{array}\right]\in \Mat_{r_2-1\times 1}\big(K\{\tau\}\big), \quad \textnormal{for}\quad l=0,1,\cdots, r_2-1$$
where,  for $l>0$, the first $l-1$ rows vanish,  and for  $l=0$ the row $\xi\Big(\theta - \theta^{(l)}\Big)\tau^k$ is omitted.
Using Proposition \ref{prop:monic_drinfeld}, we determine the explicit form of the dual \tm module under the assumption that the triangular \tm module is monic. However, by following the argument below, one can also derive the corresponding description for a non-monic triangular \tm module.

  \begin{theorem}\label{thm:postac_duala}
Let $\Upsilon(\psi_1,\psi_2,\delta)$ be a two-dimensional triangular \tm module, where
$
\psi_1=\theta+\sum_{i=1}^{r_1-1}a_i\tau^i+\tau^{r_1},
\psi_2=\theta+\sum_{i=1}^{r_2-1}b_i\tau^i+\tau^{r_2},
$
and
$
\delta=d\tau^h,
$
with $r_1>h\ge r_2>1$. Let $D:=\deg_\tau \Upsilon^\vee$, and for $k=1,\dots,D-1$ define
$
M_k:=h-(k-1)(r_2-1)-r_2.
$

Then
\begin{equation}
\label{eq:postac_duala}
\delta^\vee=\big[0,\ldots,0,1\big]
\otimes
\Bigg(
\sum_{k=1}^{D-2}\sum_{l=0}^{r_2-1}\Delta_{k,l}(\xi_{k,l})
+
\sum_{l=0}^{M_{D-1}}\Delta_{D-1,l}(\xi_{D-1,l})
\Bigg),
\end{equation}
where $\otimes$ denotes the Kronecker product of matrices, and the coefficients $\xi_{k,l}$ are defined recursively as follows:
\begin{align}
       \label{eq:rekurencja_pierwszy}
            \xi_{1,l}&=\left\{\begin{array}{ccl}
                -d & \textnormal{for} & l=M_1\\[2mm]
                -\sum\limits_{j=1}^{M_1-l} \xi_{1,l+j}b_{r_2-j}^{(l+j)}&  \textnormal{for}& l=M_1-1, M_1-2,\dots, M_{2}  \\[2mm]
                -\xi_{1,l+r_2}\big(\theta^{(l+r_2) }-\theta\big)-\sum\limits_{k=1}^{r_2-1} \xi_{1,l+r_2-k}b_{k}^{(l+r_2-k)} &  \textnormal{for}& l=M_2-1, M_{2}-2,\dots, 1,0.
            \end{array}\right.
   \end{align}
    For $k=2,3,\dots, D-2$ we have
   \begin{align}
       \label{eq:rekurencja_srodek}
            \xi_{k,l}&=\left\{\begin{array}{ccl}
               \xi^{(1)}_{k-1,M_{k-1}}& \textnormal{for} & l=M_k\\[2mm]
                 \xi_{k-1,l+r_2-1}^{(1)}-\sum\limits_{j=1}^{M_k-l} \xi_{k,l+j}b_{r_2-j}^{(l+j)}&  \textnormal{for}& l=M_k-1,\dots, M_{k+1}  \\[2mm]
               \xi_{k-1,l+r_2-1}^{(1)}-\xi_{k,l+r_2}\big(\theta^{(l+r_2)}-\theta \big)-\sum\limits_{k=1}^{r_2-1} \xi_{k,l+r_2-k}b_{k}^{(l+r_2-k)} &  \textnormal{for}& l=M_{k+1}-1, \dots, 1,0.
            \end{array}\right.
   \end{align}
   For $k=D-1$ we have
    \begin{align}
       \label{eq:rekurencja_ostatni}
            \xi_{D-1,l}&=\left\{\begin{array}{ccl}
                \xi^{(1)}_{D-2,M_{D-2}} & \textnormal{for} & l=M_{D-1}\\[2mm]
                 \xi_{D-2,l+r_2-1}^{(1)}-\sum\limits_{j=1}^{M_{D-1}-l} \xi_{D-1,l+j}b_{r_2-j}^{(l+j)}&  \textnormal{for}& l=M_{D-1}-1,\dots,1, 0
            \end{array}\right.
   \end{align}
\end{theorem}

\begin{proof} 
We use the notation introduced in the proof of Proposition \ref{prop:deg_dual_module}. The recurrence relations for $\xi_{1,l}$ follow directly from the forms of the inner biderivations in \eqref{eq:pierwszy_ciag_redukcji}. Similarly, for $k=2,3,\dots,D-2$, the recurrence relations for $\xi_{k,l}$ follow directly from the forms of the inner biderivations appearing in \eqref{eq:drugi_ciag_redukcji}, \eqref{eq:trzeci_ciag_redukcji}, and \eqref{eq:kolejne_ciagi_redukcji}.

Observe that the numbers $M_k$ are chosen so that $M_k+r_2$ is the largest exponent for which a term involving $c^{(k)}$ appears after the corresponding reduction. Finally, the recurrence relations for $\xi_{D-1,l}$ follow from the inner biderivations
$
(l,c^{(D-1)}\xi_{D-1,l})$,
$l=M_{D-1},M_{D-1}-1,\ldots,1,0.
$
Since $M_{D-1}<r_2$, there are exactly $M_{D-1}+1$ coefficients $\xi_{D-1,l}$.
A direct verification shows that, after performing the reductions by inner biderivations described in the proof of Proposition \ref{prop:deg_dual_module}, we obtain

    \begin{align*}
        \alpha_{i,1}&=\xi_{i,1}\Big(\theta - \theta^{(1)}\Big)+ \xi_{i-1,0}^{(1)} -\xi_{i,0} b_{1}, \\
        \alpha_{i,2}&=\xi_{i,2}\Big(\theta - \theta^{(2)}\Big)+ \xi_{i-1,1}^{(1)} -\xi_{i,0} b_{2}-\xi_{i,1} b_1^{(1)}\\
        \alpha_{i,3}&=\xi_{i,3}\Big(\theta - \theta^{(3)}\Big)+ \xi_{i-1,2}^{(1)} -\xi_{i,0} b_{3}-\xi_{i,1} b_2^{(1)}-\xi_{i,2} b_1^{(2)},\\
          \end{align*}
          where we set  $\xi_{0,*}=0$ for $i=1$.
          In general we have
  \begin{align*}
      \alpha_{i,l}&=\xi_{i,l}\Big(\theta - \theta^{(l)}\Big)+ \xi_{i-1,l-1}^{(1)} -\sum\limits_{j=1}^{l}\xi_{i,l-j} b_{j}^{(l-j )}\quad \textnormal{for}\quad i=1,2,\dots, D-1\\
  \end{align*}        
and 
\begin{align*}
      \alpha_{D,l+1}&=\xi_{D-1,l}^{(1)}&&\quad \textnormal{for}&\quad l=0,1,\dots, M_{D-1}\\
      \alpha_{D,l+1}&=0&&\quad \textnormal{for}&\quad l=M_{D-1}+1,M_{D-1}+2,\dots, r_2-1. 
  \end{align*} 
  Now, rearranging \eqref{eq:postac_duala_w_propozycji} with respect to $\xi_{k,l}$ we obtain the matrix $\Delta_{k,l}(\xi_{k,l})$. 
\end{proof}

{\rm Observe that, since the matrix $\Delta_{k,l}(\xi_{k,l})$ always depends on the coefficient $\xi_{k,l}$ with the same indices, we shall henceforth abbreviate $\Delta_{k,l}(\xi_{k,l})$ to simply $\Delta_{k,l}$.}

\section{Proof of the Cartier--Nishi theorem in dimension two}\label{pfthm}

In this section, we prove the main theorem. We begin by adopting the notation introduced in the previous section. For a triangular \tm module
$\Upsilon=\Upsilon(\psi_1,\psi_2,\delta),$
we consider the triangular \tm modules
$$
\Upsilon_{k,l}^\vee :=
\begin{bmatrix}
\psi_2^\vee & \Delta_{k,l} \\
0 & \psi_1^\vee
\end{bmatrix}.
$$

Then the double dual module ${\Upsilon_{k.l}^\vee}^\vee$  has the form
$$\left[\begin{array}{cc}
{\psi_2^\vee}^\vee & 0 \\
    \Delta_{k,l}^\vee  & {\psi_1^\vee}^\vee 
\end{array}\right],$$

Moreover, by \cite[Corollary 8.8]{kk25}, we have
$${\delta^\vee}^\vee=\sum\limits_{k=1}^{D-2}\sum\limits_{l=0}^{r_2-1}\Delta_{k,l}^\vee+\sum\limits_{l=0}^{M_{D-1}}\Delta_{D-1,l}^\vee.$$

Accordingly, we reduce the determination of the double dual biderivation ${\delta^\vee}^\vee$, and consequently the proof of the Cartier--Nishi theorem, to the study of an individual summand $\Delta_{k,l}^\vee$. Before proceeding, let us consider the following example.

\begin{example}
    Consider a triangular \tm module  $\Upsilon=\Upsilon(\psi_1,\psi_2,\delta)$ from Exmaple \ref{exaple736}.
Then, by Theorem \ref{thm:postac_duala}, we have
 $$\delta^\vee=[0,0,0,0,0,1]\otimes \Big( \Delta_{1,0}+  \Delta_{1,1}+
    \Delta_{1,2}+
    \Delta_{2,0}+
    \Delta_{2,1}\Big),$$
    where the coefficients $\xi_{k,l}$ satisfy the recurrence relation \eqref{rekurencje_dla_mu:7:3:6}.
    Applying the algorithm for computing the double dual, we obtain
    \begin{align*}
        \Delta_{1,0}^\vee \equiv& -\xi_{1,0}\tau^3 \mod \Derin(\psi_2,\psi_1)  \\[2mm]
         \Delta_{1,1}^\vee \equiv& -\xi_{1,1} \tau^4 - b_2^{(1)}\xi_{1,1}\tau^3  \mod \Derin(\psi_2,\psi_1)\\[2mm]
         \Delta_{1,2}^\vee \equiv&\phantom{-} \xi_{1,2}\Big(\theta^{(2)}-\theta\Big)\tau^2   \mod \Derin(\psi_2,\psi_1)\\[2mm]
         \Delta_{2,0}^\vee \equiv& -\xi_{2,0}^{(-1)}\Bigg(\tau ^5
         +b_2^{(2)}\tau ^4+b_1^{(2)} \tau ^3
         + \left(\theta^{(2)} -\theta   \right)\tau ^2\Bigg)  \mod \Derin(\psi_2,\psi_1)
        \\[2mm] 
          \Delta_{2,1}^\vee \equiv& - \xi_{2,1}^{(-1)}\Bigg(\tau ^6
          +\left(b_2 + b_2^{(3)} \right)\tau ^5
          +\left(b_1^{(3)} +b_2^{1+(2)} \right) \tau ^4\\[2mm]
          &+\left(b_2 b_1^{(2)} +\theta ^{(3)}-\theta \right)\tau ^3
   +b_2\left(\theta^{(2)} - \theta \right)\tau ^2\Bigg)  \mod \Derin(\psi_2,\psi_1).
    \end{align*}
    In addition, the notation $\xi_{k,l}^{(-1)}$ denotes the $q$-th root of $\xi_{k,l}$. Observe that, by the recurrence relations \eqref{rekurencje_dla_mu:7:3:6}, both $\xi_{2,0}^{(-1)}$ and $\xi_{2,1}^{(-1)}$ belong to the base field $K$. Furthermore, we obtain
 \begin{equation*}
        {\delta^\vee}^\vee\equiv 
        - \xi_{2,1}^{(-1)}\tau ^6 -\podwzorem{\left(\xi_{2,0}^{(-1)}+\xi_{2,1}^{(-1)}\Big(b_2 + b_2^{(3)}\Big)\right)}{=0} \tau ^5 
         -\podwzorem{\left( \xi_{1,1}+\xi_{2,0}^{(-1)}b_2^{(2)} + \xi_{2,1}^{(-1)} \Big(b_1^{(3)} +b_2^{1+(2)}\Big) \right)}{=0} \tau ^4
         \end{equation*}
         \begin{equation*}
         -\podwzorem{\left(\xi_{1,0}+\xi_{1,1}b_2^{(1)}+ \xi_{2,0}^{(-1)}b_1^{(2)}+\xi_{2,1}^{(-1)}\Big(b_2 b_1^{(2)} +\theta ^{(3)}-\theta \Big)\right)}{=0}\tau ^3
         \end{equation*}
         \begin{equation*}
         +\podwzorem{\Big(\xi_{1,2}-\xi_{2,1}^{(-1)}b_2\Big)\Big(\theta^{(2)}-\theta\Big)}{=0}\tau^2\equiv\,  d\tau^6=\delta \mod \Derin(\psi_2,\psi_1),
    \end{equation*}
    where the vanishing of the corresponding coefficients follows from the recurrence relations \eqref{rekurencje_dla_mu:7:3:6}. This computation shows that the Cartier--Nishi theorem holds in the present case.
 \end{example}
To generalize the above argument, we must determine representatives of the elements $\Delta_{k,l}^\vee$ modulo $\Derin(\psi_2,\psi_1)$. To this end, we seek to express $\Delta_{k,l}^\vee$ in the form
\begin{equation}
    \label{eq:reprezentanty_Deltavee}
    \Delta_{k,l}^\vee=\delta_{k,l}^{\mathrm{red}}+\delta_{\psi_2,\psi_1}^{(u_{k,l}(\tau))},\quad \textnormal{where}\quad \deg_\tau\delta_{k,l}^{red}<\rk\psi_1.
\end{equation}
As before, we label each case by a triple of integers $(r_1,r_2,h),$ where
$$
\psi_1=\theta+\sum_{i=1}^{r_1-1}a_i\tau^i+\tau^{r_1},\qquad
\psi_2=\theta+\sum_{j=1}^{r_2-1}b_j\tau^j+\tau^{r_2},\qquad
\delta=d\tau^h.
$$

In the tables below, we list the representatives $\delta_{k,l}^{\mathrm{red}}$ together with the polynomials $u_{k,l}(\tau)$ for selected triples $(r_1,r_2,h)$.
\begin{table}[htbp]
    \begin{tabularx}{\textwidth}{ccc@{\hspace{1.2em}}Y@{\hspace{1.2em}}Y}
\toprule
$r_1$ & $r_2$ & $h$ & $\delta_{1,2}^{\mathrm{red}}$ & $u_{1,2}(\tau)$\\
\midrule
$8$ & $3$ & $6$ & $\xi_{1,2}\Big(\theta^{(2)}-\theta\Big)\tau^{2}$ & $0$\\[0.6em]
$8$ & $3$ & $7$ & $\xi_{1,2}\Big(\theta^{(2)}-\theta\Big)\tau^{2}$ & $0$\\[0.6em]
\midrule
\end{tabularx}
\end{table}

\begin{table}[htbp]
    \begin{tabularx}{\textwidth}{ccc@{\hspace{1.2em}}Y@{\hspace{1.2em}}Y}
     \\
\toprule
$12
$ & $4$ & $10$ & $-\xi_{1,2}\sum\limits_{j=0}^{2}
b_{2+j}^{(2)}\tau^{4+j}$ &$\xi_{1,2}\tau^2$\\[0.6em]
$12$ & $4$ & $11$ & $-\xi_{1,2}\sum\limits_{j=0}^{2}
b_{2+j}^{(2)}\tau^{4+j}$ &$\xi_{1,2}\tau^2$\\[0.6em]
\midrule
$12$ & $5$ & $10$ & $-\xi_{1,2}\sum\limits_{j=0}^{2}
b_{3+j}^{(2)}\tau^{5+j}$ & $\xi_{1,2}\tau^2$\\[0.6em]
$12$ & $5$ & $11$ & $-\xi_{1,2}\sum\limits_{j=0}^{2}
b_{3+j}^{(2)}\tau^{5+j}$ & $\xi_{1,2}\tau^2$\\[0.6em]
\end{tabularx}
\caption{Representatives of $\delta_{1,2}^{\mathrm{red}}$ for various triples $(r_1,r_2,h)$.}\label{tabela:reprezentaci12}
\end{table}

\begin{table}[htbp] 
    \begin{tabularx}{\textwidth}{ccc@{\hspace{1.2em}}Y@{\hspace{1.2em}}Y}
    \\
\toprule
$r_1$ & $r_2$ & $h$ & $\delta_{2,3}^{\mathrm{red}}$ & $u_{2,3}(\tau)$\\
\midrule
\midrule
$12$ & $5$ & $10$ & $-\xi_{2,3}^{(-1)}\cdot \sum\limits_{j= 2}^{5} \delta_{\psi_2,\theta}^{\big(b_j^{(2)}\tau^{j+2}\big)}$ & $\sum\limits_{i=1}^{12} a_i \xi_{2,3}^{(i-1)} \tau^{i+2}+\xi_{2,3}^{(-1)}\sum\limits_{j=2}^{5}   b_{j}^{(2)}\tau^{j+2}$ \\[0.6em]
$12$ & $5$ & $11$ & $-\xi_{2,3}^{(-1)}\cdot \sum\limits_{j= 2}^{5} \delta_{\psi_2,\theta}^{\big(b_j^{(2)}\tau^{j+2}\big)}$ 
& $\sum\limits_{i=1}^{12} a_i \xi_{2,3}^{(i-1)} \tau^{i+2}+\xi_{2,3}^{(-1)}\sum\limits_{j=2}^{5}   b_{j}^{(2)}\tau^{j+2}$  \\[0.6em]
\midrule
$13$ & $6$ & $10$ & $-\xi_{2,3}^{(-1)}\cdot \sum\limits_{j= 3}^{6} \delta_{\psi_2,\theta}^{\big(b_j^{(2)}\tau^{j+2}\big)}$ 
& $\sum\limits_{i=1}^{13} a_i \xi_{3}^{(i-1)} \tau^{i+2}+\xi_{2,3}^{(-1)}\sum\limits_{j=3}^{6}   b_{j}^{(2)}\tau^{j+2}$ \\[0.6em]
$13$ & $6$ & $11$ & $-\xi_{2,3}^{(-1)}\cdot \sum\limits_{j= 3}^{6} \delta_{\psi_2,\theta}^{\big(b_j^{(2)}\tau^{j+2}\big)}$ 
& $\sum\limits_{i=1}^{13} a_i \xi_{2,3}^{(i-1)} \tau^{i+2}+\xi_{2,3}^{(-1)}\sum\limits_{j=3}^{6}   b_{j}^{(3)}\tau^{j+2}$ \\[0.6em]
\end{tabularx}
\caption{Representatives of $\delta_{2,3}^{\mathrm{red}}$ for various triples $(r_1,r_2,h)$.}\label{tabela:reprezentaci23}
\end{table}
Based on the examples computed above, we infer a general form for the representatives of $\Delta_{k,l}^\vee$ modulo $\Derin(\psi_2,\psi_1).$ We formulate our conclusions in the following lemma, whose proof will be given in the next section.

\begin{lemma}\label{lem:straszny}
  Let $\Upsilon=\Upsilon(\psi_1,\psi_2,\delta)$ be a reduced monic triangular \tm module with $\rk\psi_1>\rk\psi_2>1$ and 
$\delta=d\tau^{r_2+s},$
where $s\geq 0.$

    \begin{itemize}
        \item[$(i)$] If $s\in\{0,1,\dots,r_2-2\}$, then
$$
\Delta_{1,l}^\vee
\equiv
-\xi_{1,l}\sum_{j=0}^{l}
b_{r_2-l+j}^{(l)}\tau^{r_2+j}
\mod{\Derin(\psi_2,\psi_1)}
\qquad
{\text{for }} \,\,\,\,l=0,1,\dots,s.
$$

  \item[$(ii)$] If $s\in\{r_2-1,r_2,\dots,2r_2-3\}$, then the identities in part $(i)$ hold for $\Delta_{1,l}^\vee$ for all $l=0,1,\dots,r_2-2.$ Moreover,
\begin{align*}
            \Delta_{1,r_2-1}^\vee& \equiv \xi_{1,r_2-1}\Big(\theta^{(r_2-1)}-\theta\Big)\tau^{r_2-1} \mod \Derin(\psi_2,\psi_1)\\
            \Delta_{2,l}^\vee &\equiv -\xi_{2,l}^{(-1)}\cdot \sum\limits_{j= r_2-l}^{r_2} \delta_{\psi_2,\theta}^{\big(b_j^{(l-1)}\tau^{j+l-1}\big)} \mod \Derin(\psi_2,\psi_1) \quad\textnormal{for}\quad l=0,1,2,\dots, s-(r_2-1)
        \end{align*}
    \end{itemize}
\end{lemma}

We are now ready to prove Theorem \ref{thm:main}.
\begin{proof}[Proof of Theorem \ref{thm:main}]
By \cite[Corollary 8.8]{kk25}, it suffices to prove the theorem in the case where $\delta=d\tau^h$ is a monomial satisfying
$\rk \psi_2 \le h < \rk \psi_1.$

 By Proposition \ref{prop:monic_drinfeld}, we may assume, without loss of generality, that both $\psi_1$ and $\psi_2$ are monic. Writing $h=r_2+s$, we divide the proof into three cases according to the description given in Lemma \ref{lem:straszny}.

In the case $s\in\{0,1,\dots,r_2-2\}$, we have
$\delta^\vee=\sum_{l=0}^{s}\Delta_{1,l},$
where $\Delta_{1,l}=\Delta_{1,l}(\xi_{1,l})$, and the coefficients $\xi_{1,l}$ satisfy the following recurrence relations:
$$(\star_s)\quad \xi_{1,s}=-d,\quad (\star_l)\quad \xi_{1,l}= -\sum\limits_{j=1}^{s-l} \xi_{1,l+j}b_{r_2-j}^{(l+j)},\quad \textnormal{for}\quad  l=s-1, s-2,\dots, 0.$$
   Then 
   \begin{align*}
       {\delta^\vee}^\vee&= \sum\limits_{l=0}^{s}\Delta_{1,l}^\vee \equiv
       -\sum\limits_{l=0}^{s}  \xi_{1,l}\sum\limits_{j=0}^{l} b_{r_2-l+j}^{(l)}\tau^{r_2+j} \mod \Derin(\psi_2,\psi_1).
   \end{align*}
   After collecting terms according to the powers of $\tau$, we obtain
\begin{align*}
       {\delta^\vee}^\vee&
       \equiv  
      -\podwzorem{\sum\limits_{l=0}^{s}\xi_{1,l}b_{r_2-l}^{
      (l)}}{=0
      \textnormal{ from } (\star_0)} \tau^{r_2}-
      \podwzorem{\sum\limits_{l=1}^s\xi_{1,l}b_{r_2-l+1}^{(l)} }{=0,\textnormal{ from }(\star_1) } \tau^{r_2+1}\\ 
      &
      -\cdots- \podwzorem{\big( \xi_{1,s-1}b_{r_2}+\xi_{1,s}b_{r_2-1}\big)}{=0\textnormal{ from } (\star_{s-1}) }\tau^{r_2+s-1}- \podwzorem{\xi_{1,s}b_{r_2}}{=-d}\tau^{r_2+s}\\
      &\equiv d\tau^h  \mod \Derin(\psi_2,\psi_1).
   \end{align*}
   Therefore, in this case,
$\Upsilon \cong {\Upsilon^\vee}^\vee.$

   If $s\in\{r_2-1,r_2,\dots, 2r_2-3 \}$, then $\deg_\tau\Upsilon^\vee=D=3$, $M_1=s$, $M_2=s-(r_2-1)$ and $\delta^\vee=\sum\limits_{l=0}^{r_2-1}\Delta_{1,l}+\sum\limits_{l=0}^{M_2} \Delta_{2,l}$, where $\Delta_{k,l}=\Delta_{k,l}(\xi_{k,l})$ and the coefficients $\xi_{k,l}$ are determined recursively by
   \begin{align*}
       (\star_s)&\quad \xi_{1,s}=-d,\quad (\star_l)\quad \xi_{1,l}= -\sum\limits_{j=1}^{s-l} \xi_{1,l+j}b_{r_2-j}^{(l+j)},\quad \textnormal{for}\quad  l=s-1, s-2,\dots, M_2,\\
       (\star_l)&\quad \xi_{1,l}=-\xi_{1,l+r_2}\big(\theta^{(l+r_2) }-\theta\big)-\sum\limits_{j=1}^{r_2-1} \xi_{1,l+r_2-j}b_{j}^{(l+r_2-j)}\quad  \textnormal{for}\quad l=s-r_2,\dots,1 , 0,\\
       (\nabla_{M_2})&\quad \xi_{2,M_2}=\xi^{(1)}_{1,s},\quad 
       (\nabla_{l})\quad \xi_{2,l}=\xi_{1,l+r_2-1}^{(1)}-\sum\limits_{j=1}^{M_2-l} \xi_{2,l+j}b_{r_2-j}^{(l+j)}\quad  \textnormal{for}\quad  l=s-r_2,\dots,1, 0
       \end{align*}
  Then
  \begin{equation*}
      {\delta^\vee}^\vee = \sum\limits_{l=0}^{r_2-1}\Delta_{1,l}^\vee+\sum\limits_{l=0}^{M_2} \Delta_{2,l}^\vee \equiv 
      - \sum\limits_{l=0}^{r_2-2}\xi_{1,l}\sum\limits_{j=0}^{l} b_{r_2-l+j}^{(l)}\tau^{r_2+j}-\xi_{1,r_2-1}\Big(\theta-\theta^{(r_2-1)}\Big)\tau^{r_2-1}
      \end{equation*}
      \begin{equation*}
       -\sum\limits_{l=0}^{M_2}\xi_{2,l}^{(-1)}\cdot \sum\limits_{j= r_2-l}^{r_2} \delta_{\psi_2,\theta}^{(b_j^{(l-1)}\tau^{j+l-1})} \mod \Derin(\psi_2,\psi_1)
  \end{equation*}
  After collecting terms according to the powers of $\tau$, we obtain
 $ {\delta^\vee}^\vee 
       \equiv \sum\limits_{k=r_2-1}^{h} f_{k}\tau^k \mod \Derin(\psi_2,\psi_1)$,
  where
   $$f_{r_2-1}=\Big(\podwzorem{\xi_{1,r_2-1}-\sum\limits_{l=0}^{M_2}\xi_{2,l}^{(-1)}b_{r_2-l}^{(l-1)} }{=0,\textnormal{ from }(\nabla_0)}\Big)\Big(\theta^{(r_2-1)}-\theta\Big)=0.$$
   Similarly,
    \begin{align*}
       f_{r_2}&=-b_1^{(r_2-1)}\podwzorem{\sum\limits_{l=0}^{M_2}\xi_{2,j}^{(-1)}b_{r_2-j}^{(j-1)}}{=\xi_{1.r_2-1},\textnormal{ from }(\nabla_0)}
       -\sum\limits_{l=0}^{r_2-2}\xi_{1,l}b_{r_2-l}^{(l)}
       -\Big(\theta^{(r_2)}-\theta\Big)\sum\limits_{l=1}^{M_2}\xi_{2,l}^{(-1)}b_{r_2+1-l}^{(l-1)}\\
       &=\podwzorem{-\xi_{1.r_2-1}b_1^{(r_2-1)}-\sum\limits_{l=0}^{r_2-2}\xi_{1,l}b_{r_2-l}^{(l)}}{=\xi_{1,r_2}\big(\theta^{(r_2)}-\theta\big)\textnormal{ from }(\star_0)}
       -\Big(\theta^{(r_2)}-\theta\Big)\sum\limits_{l=1}^{M_2}\xi_{2,l}^{(-1)}b_{r_2+1-l}^{(l-1)}\\
       &= \Big(\theta^{(r_2)}-\theta\Big)\Big(\podwzorem{\xi_{1,r_2}-\sum\limits_{l=1}^{M_2}\xi_{2,l}^{(-1)}b_{r_2+1-l}^{(l-1)}}{=0,\textnormal{ from } (\nabla_1) }\Big)=0
  \end{align*}
  and
\begin{align*}
    f_{r_2+1}&=-b_2^{(r_2-1)}\podwzorem{\sum\limits_{l=0}^{M_2}\xi_{2,j}^{(-1)}b_{r_2-j}^{(j-1)}}{=\xi_{1.r_2-1},\textnormal{ from }(\nabla_0)}
   -b_1^{(r_2)}\podwzorem{ \sum\limits_{l=1}^{M_2}\xi_{2,l}^{(-1)}b_{r_2+1-l}^{(l-1)}}{=\xi_{1,r_2},\textnormal{ from }(\nabla_1)}\\
   &-\sum\limits_{l=0}^{r_2-2}\xi_{1,l}b_{r_2+1-l}^{(l)}
   -\Big(\theta^{(r_2+1)}-\theta\Big)\sum\limits_{l=1}^{M_2}\xi_{2,l}^{(-1)}b_{r_2+2-l}^{(l-1)}\\
   &=\podwzorem{-\sum\limits_{l=0}^{r_2}\xi_{1,l}b_{r_2+1-l}^{(l)} }{ =\xi_{1,r_2+1}\big(\theta^{(r_2+1)}-\theta\big),\textnormal{ from }(\star_1) }
   -\Big(\theta^{(r_2+1)}-\theta\Big)\sum\limits_{l=1}^{M_2}\xi_{2,l}^{(-1)}b_{r_2+2-l}^{(l-1)}\\
   &= \Big(\theta^{(r_2+1)}-\theta\Big)\Big( \podwzorem{\xi_{1,r_2+1}-\sum\limits_{l=1}^{M_2}\xi_{2,l}^{(-1)}b_{r_2+2-l}^{(l-1)} }{=0,\textnormal{ from } (\nabla_2) }\Big)=0
\end{align*}  
Proceeding analogously, we obtain
$f_{r_2+2}=f_{r_2+1}=\cdots=f_{h-1}=0.$
Finally, $f_h=-\xi_{2,M_2}^{(-1)}=-\xi_{1,s}=d,$
by $(\nabla_{M_2})$ and $(\star_s).$ Therefore,
${\delta^\vee}^\vee\equiv d\tau^h
\mod{\Derin(\psi_2,\psi_1)}.$
This completes the proof of the theorem.
\end{proof}

\section{Proof of  Lemma \ref{lem:straszny}}\label{pflemma}
In the proof of Lemma \ref{lem:straszny}, we shall apply to a \tm-module of the form 
$$ \Upsilon_{k.l}^\vee:=  \left[\begin{array}{cc}
\psi_2^\vee & \Delta_{k,l} \\
     0& \psi_1^\vee
\end{array}\right]$$
the \tm reduction algorithm described at the beginning of this paper.
All notation in the proof below is consistent with that used in Algorithm \ref{alg:doubledual}.
\begin{proof}[Proof of  Lemma \ref{lem:straszny}]
Let $\Upsilon=\Upsilon(\psi_1,\psi_2,\delta)$ be a monic reduced triangular \tm-module with $\delta=d\tau^{r_2+s}$ where $s\geq 0$. 
We divide the proof into two cases.
    
     Assume first that $s\in\{0,1,\dots, r_2-2\}$. For each $l\in\{0,1,\dots,s\}$ consider the \tm module $\Upsilon_{1,l}$.
    By Lemma \cite[Lemma 9.2]{kk25} we have
$$\Upsilon_{1,l}^\vee\cong
    \Big[\podwzorem{0,0,\dots, 0, K\tau}{r_2-1 - \textnormal{terms}} 
    \ \big|\ \podwzorem{0,0,\dots, 0, K\tau}{r_1-1 - \textnormal{terms}}  \Big].$$
   We shall determine the action   $t*c\cdot E_2  $ in the basis
     $$E_2=\Big[   \,0,\dots,0,\ \tau\,\Big|\;\boldsymbol{0}\,\Big],\quad E_1=\Big[
        \,\boldsymbol{0} \;\Big|\; \,0,\dots,0,\ \tau\, \Big]. $$    
In order to do this we reduce by the inner biderivation of the form \eqref{eq:biderywacje_wewnetrzne_cykl}:
\begin{align*}
    \delta_{\Upsilon_{1,l}^\vee,C}^{\big( \big[ V_{r_2} (c^{(1)} )\big| \boldsymbol{0}\big]   \big)}= \Bigg[  \delta_{\psi_2^\vee,C}^{\big( V_{r_2}(c^{(1)} ) \big)} \Big|\ 0,\dots,0, \sum_{x_s\in  \Delta_{1,l} }c^{(s)}  x_s \Bigg],
\end{align*}
where  $\sum\limits_{x_s\in \Delta_{1,l} }$ denotes the sum over all entries $x_s$ of the matrix  $\Delta_{1,l}$.
We obtain
\begin{align*}
    t*c\cdot E_2&= 
    \Bigg[\textnormal{ terms for }  {\psi_2^\vee}^\vee\ \Big|\, 0,\dots, 0,
    -\sum\limits_{x_s\in \Delta_{1,l} }c^{(s)}  x_s 
     \Bigg]\\
     &= \Bigg[\textnormal{ terms for }  {\psi_2^\vee}^\vee\ \Big|\ 0,\dots, 0,
    - c^{(l+1)}\xi_{1,l}^{(1)}\tau^2 +\xi_{1,l}\Bigg(\sum\limits_{s=l+1}^{r_2-1} c^{(s)}  b_{s-l}^{(l)} + c^{(l)} \Big(\theta^{(l)} - \theta\Big)\Bigg)\tau\Bigg]  
\end{align*}
The term $-c^{(k+1)} \xi_{1,l}^{(1)}\tau^2$ is reduced by means of the inner biderivation of the form \eqref{eq:biderywacje_wewnetrzne_cykl}:
\begin{align*}
    \delta_{\Upsilon_{1,l}^\vee,C}^{\big( \big[V_{r_1}(\mu) \big| \boldsymbol{0} \big]\big)}= \Bigg[ \ \boldsymbol{0}\  \Big|\  \delta_{\psi_1^\vee,C}^{\big( V_{r_1}(\mu) \big)} \Bigg],
    \quad\textnormal{for}\quad \mu=-c^{(l+1)} \xi_{1,l}^{(1)}.
\end{align*}
Then 
\begin{align*}
    t*c\cdot E_2&=   
    \Bigg[\ \textnormal{ terms for }  {\psi_2^\vee}^\vee\  \Big|\ 0,\dots, 0, 
    \sum\limits_{s=l+1}^{r_2-1} b_{s-l}^{(l)}c^{(s)}  \xi_{1,l}\tau - \sum\limits_{s=1}^{r_1}a_{s} \xi_{1,l}^{(s)}  \tau \\
    &+ c^{(l)} \xi_{1,l}\Big(\theta^{(l)} - \theta\Big)\tau \Bigg]
\end{align*}
Expressing this in the chosen basis we obtain
\begin{align*}
    \Delta_{1,l}^\vee &= \xi_{1,l}\Big(\theta^{(l)} - \theta\Big)\tau^l +
 \xi_{1,l} \sum\limits_{s=l+1}^{r_2-1} b_{s-l}^{(l)} \tau^s - \sum\limits_{s=1}^{r_1}  \xi_{1,l}^{(s)} a_{s} \tau^{s+k}. 
\end{align*}
Thus
\begin{align*}
\Delta_{1,l}^\vee+ &\Bigg(\xi_{1,l}\sum\limits_{j=0}^{l}b_{r_2-l+j}^{(l)}  \tau^{r_2+j}\Bigg) = 
\xi_{1,l}\Big(\theta^{(l)} - \theta\Big)\tau^l \\
&+\xi_{1,l}\sum\limits_{s=l+1}^{r_2-1} b_{s-l}^{(l)} \tau^s +\xi_{1,l}\sum\limits_{j=0}^{l}b_{r_2-l+j}^{(l)}\tau^{r_2+j} - \sum\limits_{s=1}^{r_1}a_{s}\xi_{1,l}^{(s)}\tau^{s+l} 
   \\
   &= \xi_{1,l}\tau^l\cdot \psi_2 - \psi_1\cdot \xi_{1,l}\tau^l\in\Derin(\psi_2,\psi_1).
\end{align*}
This completes the proof of the first case.

Now let $s\in\{r_2-1,r_2,\dots, 2r_2-3 \}$. 
The determination of representatives modulo 
 $\Derin(\psi_2,\psi_1)$ for $\Delta_{1,l}$, $l=0,1,\dots, r_2-1,$ is identical to that in the previous case. Hence, it remains to establish the claim for 
  $\Delta_{1,r_2-1}$ and $\Delta_{2,l},$ $l=0,1,\dots, s-(r_2-1)$. First, consider the \tm module 
$\Upsilon_{1,r_2-1}^\vee$. Arguing as before, we obtain that
\begin{align*}
    t*c\cdot E_2&= 
    \Bigg[\textnormal{ terms for }  {\psi_2^\vee}^\vee\ \Big|\, 0,\dots, 0,
    -\sum\limits_{x_s\in \Delta_{1,r_2-1} }c^{(s)}  x_s 
     \Bigg]\\
     &= \Bigg[\textnormal{ terms for }  {\psi_2^\vee}^\vee\ \Big|\ 0,\dots, 0,
   \xi_{1,r_2-1} c^{(r_2-1)} \Big(\theta^{(r_2-1)} - \theta\Big)\tau\Bigg] 
\end{align*}
In the chosen basis, we obtain
\begin{align*}
    \Delta_{1,r_2-1}^\vee &= \xi_{1,r_2-1}\Big(\theta^{(r_2-1)} - \theta\Big)\tau^{r_2-1}.
\end{align*}
This completes the case of $\Delta_{1,r_2-1}$. Now consider the \tm module
$\Upsilon_{2,l}^\vee$ for $l\in\{0,1,\dots, s-(r_2-1)\}$. Proceeding as before, we  obtain that
\begin{align*}
    t*c\cdot E_2&= 
    \Bigg[\textnormal{ terms for }  {\psi_2^\vee}^\vee\ \Big|\, 0,\dots, 0,
    -\sum\limits_{x_s\in \Delta_{2,l} }c^{(s)}  x_s 
     \Bigg]\\
     &= \Bigg[\textnormal{ terms for }  {\psi_2^\vee}^\vee\ \Big|\ 0,\dots, 0,
    - c^{(l+1)}\xi_{2,l}^{(1)}\tau^3 +\xi_{2,l}\Bigg(\sum\limits_{s=l+1}^{r_2-1} c^{(s)}  b_{s-l}^{(l)} + c^{(l)} \Big(\theta^{(l)} - \theta\Big)\Bigg)\tau^2\Bigg]  
\end{align*}
The term $-c^{(k+1)} \xi_{2,l}^{(1)}\tau^3$ can be reduced by means of the inner biderivation of the form \eqref{eq:biderywacje_wewnetrzne_cykl}: 
\begin{align*}
    \delta_{\Upsilon_{1,l}^\vee,C}^{\big( \big[V_{r_1}(\mu)\tau \big| \boldsymbol{0} \big]\big)}&= \Bigg[ \ \boldsymbol{0}\  \Big|\  \delta_{\psi_1^\vee,C}^{\big( V_{r_1}(\mu)\tau \big)} \Bigg],
    \quad\textnormal{for}\quad \mu=-c^{(l+1)} \xi_{2,l}^{(1)},\quad \textnormal{where}\\
     \delta_{\psi^\vee,C}^{\big( V_{r_1}(\mu)\tau\big)}&=
    \Big(\theta^{(1)}-\theta\Big)V_r(\mu)\tau+
    \Bigg[ 0,\cdots, 0, - \sum\limits_{j=1}^{r_1}a_{j}^{(1)}\mu^{(j-1)}\cdot\tau^2 + \mu \tau^{3} \Bigg]
\end{align*}
Then 
\begin{align*}
    t*c\cdot E_2&= 
    \Bigg[\textnormal{ terms for }  {\psi_2^\vee}^\vee\ \Big|\, c^{(l+1)}\xi_{2,l}^{(1)}\Big(\theta^{(1)}-\theta\Big)\tau,\,  
    c^{(l+2)}\xi_{2,l}^{(2)}\Big(\theta^{(1)}-\theta\Big)\tau
    ,\dots, \\
    &\quad c^{(l+r_1-2)}\xi_{2,l}^{(r_1-2)}\Big(\theta^{(1)}-\theta\Big)\tau,\,
    c^{(l+r_1-1)}\xi_{2,l}^{(r_1-1)}\Big(\theta^{(1)}-\theta\Big)\tau\\
    &+
   \Bigg(\xi_{2,l}\sum\limits_{s=l+1}^{r_2-1} c^{(s)}  b_{s-l}^{(l)} + \xi_{2,l}c^{(l)} \Big(\theta^{(l)} - \theta\Big) - \sum\limits_{j=1}^{r_1}a_{j}^{(1)}\xi_{2,l}^{(j)} c^{(l+j)} \Bigg)\tau^2 \Bigg].
\end{align*}
For each $k=1,2,\dots r_1-2$ the term $ c^{(l+k)}\xi_{2,l}^{(k)}(\theta^{(1)}-\theta)\tau$ is reduced by means of the following inner biderivation:
\begin{align*}
    \delta_{\Upsilon_{1,l}^\vee,C}^{\big( \big[W_{k,r_1}(\mu_k)\big| \boldsymbol{0} \big]\big)}&= \Bigg[ \ \boldsymbol{0}\  \Big|\  \delta_{\psi_1^\vee,C}^{\big( W_{k,r_1}(\mu_k)\big)} \Bigg],
    \quad\textnormal{for}\quad \mu_k=c^{(l+k)} \xi_{2,l}^{(k)}(\theta^{(1)}-\theta),\quad \textnormal{where}\\
     \delta_{\psi^\vee,C}^{\big( W_{k,r_1}(\mu_k)\big)}&=
    \Bigg[ \podwzorem{0,\cdots,0}{(k-1)\textnormal{ terms}}, \mu_k\tau,0,\cdots,0, - \sum\limits_{j =k+1}^{r_1}a_{j} \mu_k^{(j-k-1)}\tau \Bigg]
\end{align*}
As a consequence, we obtain
\begin{align*}
    t*c\cdot E_2&= 
    \Bigg[\textnormal{ terms for }  {\psi_2^\vee}^\vee\ \Big|\,   0
    ,\dots, 0,\,   c^{(l+r_1-1)}\xi_{2,l}^{(r_1-1)}\Big(\theta^{(1)}-\theta\Big)\tau\\
    &+ \sum\limits_{k=1}^{r_1-2}\sum\limits_{j =k+1}^{r_1}a_{j} c^{(l+j-1)}\xi_{2,l}^{(j-1)}\Big(\theta^{(j-k)}-\theta^{(j-k-1)}\Big) \tau\\
    &+
   \Bigg(\xi_{2,l}\sum\limits_{s=l+1}^{r_2-1} c^{(s)}  b_{s-l}^{(l)} + \xi_{2,l}c^{(l)} \Big(\theta^{(l)} - \theta\Big) - \sum\limits_{j=1}^{r_1}a_{j}^{(1)}\xi_{2,l}^{(j)} c^{(l+j)} \Bigg)\tau^2 \Bigg].
\end{align*}
Note that, after collecting like terms, we obtain the identity
$$\sum\limits_{k=1}^{r_1-2}\sum\limits_{j =k+1}^{r_1}a_{j} c^{(l+j-1)}\xi_{2,l}^{(j-1)}\Big(\theta^{(j-k)}-\theta^{(j-k-1)}\Big)+ \xi_{2,l}^{(r_1-1)}\Big(\theta^{(1)}-\theta\Big) =\sum\limits_{i=2}^{r_1} a_i c^{(l+i-1)}\xi_{2,l}^{(i-1)}\Big(\theta^{(i-1)}-\theta\Big).
$$
Thus
\begin{align*}
    t*c\cdot E_2&= 
    \Bigg[\textnormal{ terms for }  {\psi_2^\vee}^\vee\ \Big|\,   0
    ,\dots, 0,\,  \sum\limits_{i=2}^{r_1} a_i c^{(l+i-1)}\xi_{2,l}^{(i-1)}\Big(\theta^{(i-1)}-\theta\Big) \tau\\
    &+
   \Bigg(\xi_{2,l}\sum\limits_{s=l+1}^{r_2-1} c^{(s)}  b_{s-l}^{(l)} + \xi_{2,l}c^{(l)} \Big(\theta^{(l)} - \theta\Big) - \sum\limits_{j=1}^{r_1}a_{j}^{(1)}\xi_{2,l}^{(j)} c^{(l+j)} \Bigg)\tau^2 \Bigg].
\end{align*}
Denote
$$A:=\Bigg(\xi_{2,l}\sum\limits_{s=l+1}^{r_2-1} c^{(s)}  b_{s-l}^{(l)} + \xi_{2,l}c^{(l)} \Big(\theta^{(l)} - \theta\Big) - \sum\limits_{j=1}^{r_1}a_{j}^{(1)}\xi_{2,l}^{(j)} c^{(l+j)} \Bigg).$$
We then reduce the term $A\tau^2$ using the following inner biderivation:
\begin{align*}
    \delta_{\Upsilon_{1,l}^\vee,C}^{\big( \big[V_{r_1}(A)\big| \boldsymbol{0} \big]\big)}&= \Bigg[ \ \boldsymbol{0}\  \Big|\  \delta_{\psi_1^\vee,C}^{\big( V_{r_1}(A)\big)} \Bigg],
    \quad\textnormal{where}\\
     \delta_{\psi^\vee,C}^{\big( V_{r_1}(A)\tau\big)}&=
    \Bigg[ 0,\cdots, 0, - \sum\limits_{i=1}^{r_1}a_{i} A^{(i-1)}\cdot\tau+ A \tau^{2} \Bigg]
\end{align*}
thus obtaining the final form
\begin{align*}
    t*c\cdot E_2&= 
    \Bigg[\textnormal{ terms for }  {\psi_2^\vee}^\vee\ \Big|\,   0
    ,\dots, 0,\, \sum\limits_{i=2}^{r_1} a_i c^{(l+i-1)}\xi_{2,l}^{(i-1)}\Big(\theta^{(i-1)}-\theta\Big) \tau \\
    &+ \sum\limits_{i=1}^{r_1} a_i \xi_{2,l}^{(i-1)}\sum\limits_{s=l+1}^{r_2-1} c^{(s+i-1)}  b_{s-l}^{(l+i-1)}\tau\\
    &+\sum\limits_{i=1}^{r_1} a_i \xi_{2,l}^{(i-1)}c^{(l+i-1)} \Big(\theta^{(l+i-1)} - \theta^{(i-1)}\Big)\tau
    -\sum\limits_{i=1}^{r_1} a_i \sum\limits_{j=1}^{r_1}a_{j}^{(i)}\xi_{2,l}^{(j+i-1)} c^{(l+j+i-1)}\tau
      \Bigg].
\end{align*}
Moreover, we have
\begin{align*}
    \sum\limits_{i=2}^{r_1} a_i c^{(l+i-1)}\xi_{2,l}^{(i-1)}\Big(\theta^{(i-1)}-\theta\Big) \tau &+ \sum\limits_{i=1}^{r_1} a_i \xi_{2,l}^{(i-1)}c^{(l+i-1)} \Big(\theta^{(l+i-1)} - \theta^{(i-1)}\Big)\tau\\
    &= \sum\limits_{i=1}^{r_1} a_i \xi_{2,l}^{(i-1)}c^{(l+i-1)} \Big(\theta^{(l+i-1)} - \theta\Big)\tau,
\end{align*}
Therefore, in the chosen basis, we obtain
\begin{align*}
    \Delta_{2,l}^\vee&= 
      \sum\limits_{i=1}^{r_1} a_i \xi_{2,l}^{(i-1)}\sum\limits_{s=l+1}^{r_2-1}   b_{s-l}^{(l+i-1)}\tau^{s+i-1}\\
    &+\sum\limits_{i=1}^{r_1} a_i \xi_{2,l}^{(i-1)} \Big(\theta^{(l+i-1)} - \theta\Big)\tau^{l+i-1}
    -\sum\limits_{i=1}^{r_1} a_i \sum\limits_{j=1}^{r_1}a_{j}^{(i)}\xi_{2,l}^{(j+i-1)} \tau^{l+j+i-1}.
\end{align*}
Let us set $U:=\sum\limits_{i=1}^{r_1} a_i \xi_{2,l}^{(i-1)} \tau^{l+i-1}$ and
$V:=\xi_{2,l}^{(-1)}\sum\limits_{j=r_2-l}^{r_2}   b_{j}^{(l-1)}\tau^{j+l-1}.$ 

Then 
\begin{align*}
    \Delta_{2,l}^\vee&= U\psi_2-\psi_1 U-
      \sum\limits_{i=1}^{r_1} a_i \tau^{i} V. 
\end{align*}
Finally we obtain
\begin{align*}
    \Delta_{2,l}^\vee+ \xi_{2,l}^{(-1)}\cdot \sum\limits_{j= r_2-l}^{r_2} \delta_{\psi_2,\theta}^{\big(b_j^{(l-1)}\tau^{j+l-1}\big)}&= 
     \Delta_{2,l}^\vee+ \delta_{\psi_2,\theta}^{(V)}=U\psi_2-\psi_1 U- \sum\limits_{i=1}^{r_1} a_i \tau^{i} V + V\psi_2-\theta V\\
     &=     U\psi_2-\psi_1 U+V \psi_2- \psi_1 V\in\Derin(\psi_2,\psi_1), 
\end{align*}
thereby completing the proof of $(ii)$. 

\end{proof}

We note that the specific forms of the polynomials $U$ and $V$ are not unexpected, since they are suggested by empirical computations, several examples of which are listed in Tables \ref{tabela:reprezentaci12} and  \ref{tabela:reprezentaci23}.

\begin{remark}
Our approach to establishing the isomorphism
${\Upsilon^\vee}^\vee \cong \Upsilon$
provides an explanation for the regularities observed in Section~\ref{poly}. Indeed, observe that   equality~\eqref{eq:reprezentanty_Deltavee}, together with the formula
$$
{\delta^\vee}^\vee=\sum_{k,l}\Delta_{k,l}^\vee,
$$
implies that
 \begin{align*}
     {\delta^\vee}^\vee&= \sum\limits_{k,l}\delta_{k,l}^{\mathrm{red}}+\sum\limits_{k,l}\delta_{\psi_2,\psi_1}^{\big(u_{k,l}(\tau)\big)}.
 \end{align*}
 If we consider the \tm module determined by the triple of integers $(r_1,r_2,h)$, then, according to the notation introduced in Section~\ref{poly}, we have
$$u(r_1,r_2,h)=\sum_{k,l}u_{k,l}(\tau).$$
From the description of the dual \tm module given in Theorem~\ref{thm:postac_duala}, it follows that for $h\leq 2(r_2-1)$ only components of the form $\Delta_{1,l}$ with $l<r_2-1$ occur. Therefore, by the proof of Lemma~\ref{lem:straszny}, we obtain
$$
u(r_1,r_2,h)=\sum_{l}u_{1,l}(\tau)
           =\sum_{l}\xi_{1,l}\tau^l.$$
Hence, in this case, $u(r_1,r_2,h)$ is independent of $r_1$ and $r_2$, which explains observation~\textbf{a.} from Section~\ref{poly}.

To justify observation~\textbf{b.}, let us note that in the case
$2(r_2-1)<h\leq 3(r_2-1),$
the dual \tm module contains only components of the form $\Delta_{1,l}$ for
$l=0,1,\dots,r_2-1$ and $\Delta_{2,l}$ for $l<r_2-1$. Therefore, it follows from the proof of Lemma~\ref{lem:straszny} that
$$u(r_1,r_2,h) =
\sum_{l}\xi_{1,l}\tau^l
+
\sum_{l}
\left(
\sum_{i=1}^{r_1}
a_i \xi_{2,l}^{(i-1)} \tau^{l+i-1}
+
\xi_{2,l}^{(-1)}
\sum_{j=r_2-l}^{r_2}
b_{j}^{(l-1)}\tau^{j+l-1}
\right).
$$
Since the sum
$\sum_{l}\sum_{i=1}^{r_1}a_i \xi_{2,l}^{(i-1)} \tau^{l+i-1}$
depends on $r_1$ and on the coefficients of $\psi_1$, it gives rise to the recurrence relations observed in observation~\textbf{b.} of Section~\ref{poly}.

\end{remark}

\section{Discussion and future directions}\label{disc}
Our main theorem on Cartier--Nishi duality for two-dimensional triangular \tm modules allowing for $\mathrm{ALD}$-biderivations extends the class of modules for which the duality is known to hold. Recall that in \cite{kk25} we proved the Cartier--Nishi theorem for triangular \tm modules of arbitrary dimension, but under a more restrictive degree condition, namely for those allowing for $\mathrm{LD}$-biderivations. The latter case appears to be more tractable, since the relevant isomorphisms are induced by constant polynomials. Both proofs rely crucially on the existence of a reduced form, established in \cite[Theorem 4.1]{kk25}. 
If one attempts to work with non-reduced forms of \tm modules, the formulas appearing in Lemma~\ref{lem:straszny} become so complicated as to be practically intractable. For \tm modules that do not allow for $\mathrm{ALD}$-biderivations, the computational complexity increases substantially: one obtains highly irregular formulas and, moreover, must work over a finite extension of the field of definition of the \tm module.

Despite these obstacles, we believe that Cartier--Nishi duality should hold in a much broader setting, possibly for all dimensions. This belief motivates the general conjectures stated above. However, establishing such a result will likely require a modification of approach  developed in the present paper.

The following example is instructive.

\begin{example}
Consider a triangular \tm module
$\Upsilon=\Upsilon(\psi_1,\psi_2,\delta),$
where
$$\psi_1=\theta+\sum_{i=1}^7 a_i\tau^i+\tau^8,\qquad
\psi_2=\theta+b_1\tau+b_2\tau^2+\tau^3,
\qquad\text{and}\qquad
\delta=d\tau^7.$$

This module does not allow for $\mathrm{ALD}$-biderivations. Then, by Theorem~\ref{thm:postac_duala},
$$
\delta^\vee=[0,0,0,0,0,1]\otimes
\Big(
\Delta_{1,0}
+\Delta_{1,1}
+\Delta_{1,2}
+\Delta_{2,0}
+\Delta_{2,1}
+\Delta_{2,2}
+\Delta_{3,0}
\Big).
$$

By computing the terms $\Delta_{k,l}^\vee$ using the \tm reduction algorithm, we make the following observations:
\begin{itemize}
    \item[\textbf{a.}] The representatives of the biderivations $\Delta_{1,l}^{\vee}$ for $l\in{0,1,2}$ and $\Delta_{2,l}^{\vee}$ for $l\in{0,1}$ coincide with those obtained in Lemma~\ref{lem:straszny}.

\item[\textbf{b.}] If we attempt to determine representatives of $\Delta_{2,2}^{\vee}$ and $\Delta_{3,0}^{\vee}$ in the same manner, it turns out that:
\begin{itemize}
\item We are forced to extract $q$-th roots of elements that need not be $q$-th powers in the base field (K), for example $\theta^{(-1)}$. This necessitates passing to a finite extension of $K$. It is worth noting that the direct computation of $u(8,3,7)$ did not require such a field extension.

\item With representatives determined in this manner, it becomes difficult to carry out a proof analogous to that of Theorem~\ref{thm:main}.
\end{itemize}
\end{itemize}
This suggests that the decomposition into $\Delta_{2,2}^{\vee}$ and $\Delta_{3,0}^{\vee}$ may not be the most natural one. In the present case, this issue can be resolved by considering $\Delta_{2,2}^{\vee}$ and $\Delta_{3,0}^{\vee}$ jointly rather than separately, that is, by computing
$\bigl(\Delta_{2,2}+\Delta_{3,0}\bigr)^{\vee}.$
One then obtains the following relation:
\begin{align*}
       \Big(\Delta_{2,2}+\Delta_{3,0}\Big)^\vee \equiv \xi_{2,2}^{(-1)}\cdot \sum\limits_{j=1}^3\delta_{\psi_2,\theta }^{\big(b_j^{(1)}\tau^{j+1}\big) } \mod \Derin(\psi_2,\psi_1),
    \end{align*}
    which corresponds to the representative obtained in Lemma~\ref{lem:straszny}. Consequently, we may repeat the argument used in the proof of Theorem~\ref{thm:main} to obtain an isomorphism ${\Upsilon^\vee}^\vee \cong \Upsilon.$
 \end{example}  

Motivated by this example, and by repeating the computations from the proof of Lemma~\ref{lem:straszny} together with the argument used in the proof of Theorem~\ref{thm:main}, one might hope to extend the Cartier--Nishi conjecture to the case $(r_1,r_2,h)$ with
$h = 3(r_2-1)+1.$

 This example raises the hope that the method proposed in this paper may also be adapted to the general case. However, for triples $(r_1,r_2,h)$ with $h>3(r_2-1)+1,$
that is, when terms of the form $\Delta_{3,l}$ with $l>0$ appear in the description of the dual \tm module, the same line of reasoning no longer works.

As before, the biderivations $\Delta_{1,l}$ and $\Delta_{2,l}$ for $l<r_2-1$ admit the same representatives as those appearing in Lemma~\ref{lem:straszny}. However, considering $\Delta_{2,r_2-1}$, either separately or together with $\Delta_{3,0}$ and the remaining terms $\Delta_{3,l}$, no longer yields the desired outcome.

\end{document}